\documentclass[11pt]{article}
\usepackage[numbers]{natbib}
\usepackage{marsden_article}
\usepackage[all]{xy}
\usepackage{amssymb,amsfonts}

\newtheoremstyle{obs}% name
{3pt}%      Space above
{3pt}%      Space below
{}%         Body font
{}%         Indent amount (empty = no indent, \parindent = para indent)
{\bfseries}% Thm head font
{.}%        Punctuation after thm head
{.5em}%     Space after thm head: " " = normal interword space;
{}%         Thm head spec (can be left empty, meaning `normal')
\theoremstyle{obs}
\newtheorem{remark}[theorem]{Remark}
\newtheorem{example}[theorem]{Example}

\newcommand{\norm}[1]{\left\lVert#1\right\rVert}
\newcommand{\rvline}{\hspace*{-\arraycolsep}\vline\hspace*{-\arraycolsep}}
\usepackage{authblk}

\newcommand\xqed[1]{%
	\leavevmode\unskip\penalty9999 \hbox{}\nobreak\hfill
	\quad\hbox{#1}}

\newcommand\demo{\xqed{$\triangle$}}

\title{\bf Retraction maps: a seed of geometric integrators}
\author[1]{M.\ Barbero-Li\~n\'an}
\author[2]{D.\ Mart\'{\i}n de Diego}
\affil[1]{\small Departamento de Matem\'atica Aplicada, Universidad Polit\'ecnica de Madrid, Av. Juan de Herrera 4, 28040 Madrid, Spain, }
\affil[2]{\small Instituto de Ciencias Matem\'aticas (CSIC-UAM-UC3M-UCM), C/Nicol\'as Cabrera 13-15, 28049 Madrid, Spain}

\begin{document}

%	\title[Retraction maps: geometric integrators]{Retraction maps: a seed of geometric integrators}

%	\author*[1]{\fnm{Mar\'ia} \sur{M.\ Barbero-Li\~n\'an}}\email{m.barbero@upm.es}
	
%	\author[2]{\fnm{David} \sur{Mart\'{\i}n de Diego}}\email{david.martin@icmat.es}
%	\equalcont{The authors contributed equally to this work.}

	%\affil*[1]{\orgdiv{Departamento de Matem\'atica Aplicada}, \orgname{Universidad Polit\'ecnica de Madrid}, \orgaddress{\street{Av. Juan de Herrera 4}, \city{Madrid}, \postcode{28040}, \country{Spain}}}
	
	%\affil[2]{\orgname{Instituto de Ciencias Matem\'aticas (CSIC-UAM-UC3M-UCM)}, \orgaddress{\street{C/Nicol\'as Cabrera 13-15}, \city{Madrid}, \postcode{28049}, \country{Spain}}}

	\maketitle

\abstract{The classical notion of retraction map used to approximate geodesics is extended and rigorously defined to become a powerful tool to construct geometric integrators and it is called discretization map. Using the geometry of the tangent and cotangent bundles, we are able to tangently and cotangent lift such a map so that these lifts inherit the same properties as the original one and they continue to be discretization maps. In particular, the cotangent lift of a discretization map is a natural symplectomorphism, what plays a key role for constructing geometric integrators and symplectic methods.  As a result, a wide range of (higer-order) numerical  methods are recovered and canonically constructed by using different discretization maps, as well as some operations with Lagrangian submanifolds.

\vspace{4mm}


\vspace{4mm}


}

%\tableofcontents

\section{Introduction}

The notion of retraction map is an essential tool in different research areas like optimization theory,  numerical analysis, interpolation (see \cite{AbMaSeBookRetraction} and references therein). 

In optimization theory, the goal is to find a value $x$ in a differentiable manifold $M$ such that $f(x)$ is the minimum of a real-valued function $f: M\rightarrow {\mathbb R}$. In the case that $M$ is a linear space, as ${\mathbb R}^n$ equipped with the standard inner product, the notions of gradient or Hessian of the function $f$ are properly defined and give us useful local information to localize the possible candidates to minimize $f$. Moreover, gradient  descent  or  Newton’s  method can also be used to search for a solution. 

Riemannian geometry allows us to introduce similar concepts to gradient and Hessian in a differentiable manifold paving  the  way  for  optimization. But we need another important ingredient: how to move on a manifold. In Riemannian geometry this notion is given by the exponential map. On a Riemannian manifold $(M, g)$ (or more generally a semi-Riemannian manifold) we can define  ${\rm exp}_x: T_x M \to M$ the Riemannian exponential at the point $q$. As mentioned for instance in~\cite{doCarmo},
$
{\rm exp}_x(\xi) = \sigma(1),
$
for $\xi \in T_x M$, where $\sigma: [0, 1] \to M$ is the unique geodesic in $M$ with initial velocity $\xi$, that is, $\sigma(0) = x$ and $\dot{\sigma}(0) =\xi$. Moreover, there exist open subsets ${\mathcal U} \subseteq T_x M$ and $U \subseteq M$, with ${\mathcal U}$ starshaped about $0_x \in {\mathcal U}$ and $x \in U$, such that
$
{\rm exp}_x: {\mathcal U} \to U
$
is a diffeomorphism and 
$
{\rm exp}_x(t\xi) = \sigma(t)$ and $T_{0_x}{\rm exp}_x = Id_{T_x M}$.
However, only in simple examples it is possible to explicitly compute the exponential map of a Riemannian manifold. Therefore, efficient approximations of geodesics are  crucial for designing algorithms on manifolds. Here is where retraction maps play an important role (see \cite{2012AbsilMalick,AbMaSeBookRetraction} and references therein). 

Roughly speaking, retraction maps provide a way to select a smooth curve on a differentiable manifold given an initial position and velocity. Such a curve is an approximation of the Riemannian exponential map. 
More specifically, a retraction is typically defined as a local $C^1$-map $R_x: U_x\subset T_xM\rightarrow M$ such that $R_x(0_x) = x$   and 
$\left.\frac{d}{dt}\right\vert_{t=0}R_x(t\xi)=T_{0_x}R_x(\xi)= \xi$  for all $\xi \in T_x M$, where we use the identification $T_{0_x}T_x M \equiv T_xM$ and $T_{0_x}R_x: T_{0_x}(T_xM)\rightarrow T_xM$ denotes the tangent map of $R_x$ at $0_x$ (see \cite{AbMa}). 
Observe that, since we are using first order approximations, this definition is independent of the initial Riemannian metric . However, for second or higher order retractions the particular Riemannian metric does play a role.
The property $\left.\frac{d}{dt}\right\vert_{t=0}R_x(t\xi)=\xi$ implies that
$d_g(R_x(t\xi),\sigma(t))=O(t^2)$, where $d_g$ denotes the Riemannian distance  (see \cite{1986Shub}).

For our purposes we will need a more general definition of a retraction map. We construct a discretization map $R_d:   U\subset TM\rightarrow M\times M$ in Definition \ref{def:DiscreteMap2}, where  the image  of  $\xi\in U$ is now two ``nearby" points of $M$. We understand such a map as a discretization of the tangent bundle because $TM$ is locally diffeomorphic to two copies of the  manifold $M$. As they can be related to retraction maps, they are denoted by $R_d$, where the subscript $d$ stands for discretization.

As an example, if we have a Riemannian manifold $(M, g)$, with associated exponential map $exp$, then a discretization map is
$$R_d(\xi)=\left(exp_{\tau_M(\xi)} \left(-\frac{1}{2}\xi\right), exp_{\tau_M(\xi)}\left(\frac{1}{2}\xi\right)\right),$$ where $\tau_M: TM\rightarrow M$ is the canonical projection of the tangent bundle. This particular map $R_d$ applied in Equation~\eqref{first} will lead to the implicit midpoint method on Euclidean spaces. We precisely discuss the properties
 of these discretization maps in Section~\ref{Sec:Rd}.  

In numerical analysis, if we have a vector field $X$ on $M$, that is,  a section $X: M\rightarrow TM$ (that is, $\tau_M\circ X=Id_{M}$), and we want to find a numerical approximation of the integral curves, an idea is to use a discretization map and consider the following first order discrete equation: 
\begin{equation}\label{first}
h X\left(\tau_M\left( R_d^{-1}(x_k, x_{k+1})\right)\right)=R_d^{-1}\left(x_k, x_{k+1}\right)\,.
\end{equation} We prove in Proposition~\ref{Prop:inv_Rd} that $R_d$ is a local diffeomorphism and the inverse map can be computed.
Given an initial condition $x_0$, we might be able to solve the implicit system~\eqref{first} to find a sequence $\{x_k\}$ which is an approximation of $\{x(kh)\}$, where $x(t)$ is the integral curve of $X$ with initial condition $x_0$ and $h$ is the time step. 
For instance, if $M$ is the  vector space ${\mathbb R}^n$ and $R_d(x, v)=\left(x-\frac{v}{2}, x+\frac{v}{2}\right)$, then Equation~\eqref{first} becomes
\[
\frac{x_{k+1}-x_k}{h}=X \left(\frac{x_{k}+x_{k+1}}{2}\right)\, .
\]

Our main interest in this article consists of designing numerical methods for second order differential equations (SODEs) and mainly for Hamilton's
equations.
For instance, a second order differential equation $\ddot{x}=f(x, \dot{x})$ is geometrically represented by a special vector field 
\[
\Gamma(x,\dot{x})=\dot{x}\frac{\partial}{\partial x}+ f(x, \dot{x})\frac{\partial}{\partial \dot{x}}\; ,
\]
which is now defined on the tangent bundle $TM$ of $M$ \cite{AbMa}. These vector fields are called SODEs.    

On the other hand, it is well-known that the classical Hamilton's equations are defined on the cotangent bundle $T^*M$ of the manifold $M$. Therefore, we face the problem of how, given a discretization map on $M$, we can lift it to the tangent and cotangent bundles. Besides, we define in Proposition 2.5 adjoint discretization maps by inversion with the objective to construct symplectic symmetric numerical methods of higher-order.

In Section~\ref{Sec:TliftRd}, we lift a discretization map on a manifold to the tangent bundle using the canonical involution. This tangent lift makes possible to define geometric discretizations of SODEs in Section~\ref{Sec:AllSODE}.

In Section~\ref{Sec:coTliftRd} we lift a discretization map on a manifold to the cotangent bundle using well-known constructions from symplectic geometry. We show in Section~\ref{Sec:DualLift} that this cotangent lift is nothing else than the dual construction of the above-mentioned tangent lift. Moreover, it is essential to prove that the cotangent lift of a discretization map is always a symplectomorphism because that makes possible to construct symplectic integrators for Hamilton's equations and Euler-Lagrange equations in Section~\ref{Sec:GeomInt}.

In Section~\ref{Sec:ExliftRd} we carefully work out a few examples of the discretization maps on different manifolds. 

Section~\ref{Sec:GeomInt} describes how to obtain numerical methods for Euler-Lagrange equations and Hamilton's equations using the tools described in the previous sections. In particular, in Section~\ref{Sec:discreteVarCalc} we compare the geometric integrators from Section~\ref{Sec:GeomIntHam} with the theory of discrete variational calculus~\cite{MW_Acta}. When the symplectic numerical methods in Section~\ref{Sec:GeomIntHam} are understood as Lagrangian submanifolds, we can start to compose Lagrangian submanifolds coming from different discretization maps to construct general symplectic methods for Hamilton's equation in Section~\ref{Sec:Compose}. We study how to define higher-order geometric methods by composing symmetric symplectic methods in Section~\ref{Sec_higherorder}.

Along the paper we show how well-known geometric methods (Newmark, St\"ormer-Verlet, etc) are obtained using the new tools described here. Hence,
the work developed in this paper opens the path to define, even higher-order, geometric integrators for more complex mechanical systems that may include forced systems, system with constraints, optimal control problems, Dirac systems, etc. We describe specific future research lines in Section~\ref{Sec:future}.

\section{Retraction maps} \label{Sec:Rd}

A retraction map plays the role of generalizing the linear-search methods in Euclidean spaces to general manifolds. On a manifold with nonzero curvature to move along the tangent line does not guarantee that the motion stays on the manifold. The retraction map provides the tool to define the notion of moving in a direction of a tangent vector while staying on the manifold. That is why retraction maps have been widely used to construct numerical integrators of ordinary differential equations, since it allows us to move from a point and a velocity to one nearby point so that the differential equation can be discretized.

The first notion of retraction that appears in the literature can be found in~\cite{1931Borsuk} from a topological viewpoint. Later on, the notion of retraction map as defined below is used to define Newton's method on Riemannian manifolds~\cite{1986Shub,2002Adler}.

\begin{definition}\label{def-RetractMap} A \textbf{retraction map} on a manifold $M$ is a smooth mapping $R$ from the tangent bundle $TM$ onto $M$. Let $R_x$ denote the restriction of $R$ to $T_xM$, the following properties are satisfied:
\begin{enumerate}
	\item $R_x(0_x)=x$ where $0_x$ denotes the zero element of the vector space $T_xM$.
	\item With the canonical identification $T_{0_x}T_xM\simeq T_xM$, $R_x$ satisfies \begin{equation}\label{eq-DefRetract-prop}
	{\rm D}R_x(0_x)=T_{0_x}R_x={\rm Id}_{T_xM},
	\end{equation}
	where ${\rm Id}_{T_xM}$ denotes the identity mapping on $T_xM$.
\end{enumerate}
\end{definition}

The condition~\eqref{eq-DefRetract-prop} is known as \textbf{local rigidity condition} since, given $\xi\in T_xM$,  the curve $\gamma_\xi(t)=R_x(t\xi)$ has $\xi$ as tangent vector at $x$, i.e.
\[
\dot{\gamma}_\xi(t)= \langle DR_x(t\xi), \xi\rangle\;   \hbox{  and, in consequence,   } \dot{\gamma}_\xi(0)= {\rm Id}_{T_xM}(\xi)=\xi\; .
\]
This notion connects with the geometric interpretation of the exponential map ${exp}$ on Riemannian manifolds given in~\cite[Chapter 3.2]{doCarmo}. Therefore the image of $\xi$ through the exponential map is a point on the Riemannian manifold obtained by moving along a geodesic a length equal to the norm of $\xi$ starting with the velocity $\xi/\|\xi\|$, that is, 
\[
exp_x(\xi)=\sigma (\|\xi\|)\; ,
\]
where $\sigma$ is the unit speed geodesic such that $\sigma(0)=x$ and $\dot{\sigma}(0)=\xi/\|\xi\|$. 

 Remember that the exponential map is a typical example of a retraction map. With all that in mind we are able to generalize the property of local rigidity in Definition~\ref{def-RetractMap} that allows a discretization of the tangent bundle of the configuration manifold opening a new path to construct numerical integrators.

After studying the contribution given  in~\cite{2009CuellPatrick,3MMM} we define a generalization of the retraction map in Definition~\ref{def-RetractMap}. Given a point and a velocity, we obtain two nearby points that are not necessarily equal to the initial base point. As discussed in the sequel, numerical methods will be recovered from this new map.

\begin{definition} \label{def:DiscreteMap2} A map 
	$R_d\colon U\subset TM\rightarrow M\times M$ given by \begin{equation*}
	R_d(x,v)=(R^1(x,v),R^2(x,v)),
	\end{equation*} 
	where  $U$ is an open neighborhood of the zero section $0_x$ of $TM$, 
	defines a {\bf discretization map on $M$} if it satisfies 
	\begin{enumerate}
		\item $R_d(x,0)=(x,x)$,
		\item $T_{0_x}R^2_x-T_{0_x}R^1_x\colon T_{0_x}T_xM\simeq T_xM\rightarrow T_xM$ is equal to the identity map on $T_xM$ for any $x$ in $M$, where $R^a_x$ denotes the restrictions of $R^a$, $a=1,2$, to $T_xM$.
	\end{enumerate}
\end{definition}
If $R^1(x,v)=x$, the two properties in Definition~\ref{def:DiscreteMap2} guarantee that the both properties in Definition~\ref{def-RetractMap} are satisfied by $R^2$. Thus, as mentioned, Definition~\ref{def:DiscreteMap2} generalizes Definition~\ref{def-RetractMap}. 

\begin{proposition}\label{Prop:inv_Rd} Let $R_d$ be an discretization map on $M$,
	$R_d$ is a local diffeomorphism from some neighborhood of the zero  section of  $TM$. 
	\end{proposition}
\begin{proof}
	Let $(x^i)$ be local coordinates for $M$ centered at $x$, and $(x^i, v^i)$ be the corresponding induced  coordinates on $TM$ centered at $0_x$. By the definition of the discretization map, the Jacobian matrix of $R_d$ at $(x^i,0)$ is locally written as
	\[
	\left(
	\begin{array}{cc}
	{\rm Id}&
	\frac{\partial R^1}{\partial v}(x, 0)\\ {\rm Id} &
	\frac{\partial R^2}{\partial v}( x, 0)
	\end{array}
	\right)	\, ,
	\] where ${\rm Id}$ denotes the identity matrix. Note that the 
	regularity of that Jacobian matrix is equivalent to the invertibility  of the following matrix
	\[
	\left(
	\begin{array}{cc}
	{\rm Id}&
	\frac{\partial R^1}{\partial v}(x, 0)\\0
	&\frac{\partial R^2}{\partial v}(x, 0)-\frac{\partial R^1}{\partial v}(x, 0)\end{array}
	\right)=\left(
	\begin{array}{cc}
	{\rm Id}&
	\frac{\partial R^1}{\partial v}(x, 0)\\0&
	{\rm Id}\end{array}
	\right)	 \, ,
	\]
	due to the property 2 in Definition~\ref{def:DiscreteMap2}.
	Therefore, the inverse function theorem guarantees that $R_d$  is a local diffeomorphism from some neighborhood of the identity section to its image.	
	\end{proof}

	There is a general and interesting  way to obtain discretization maps from the usual retraction maps. The following result is very useful for the Examples~\ref{Ex:Sphere} and~\ref{Ex:cayley}.

	\begin{proposition}\label{prop:buildExtRetraction} Let $R\colon TM \rightarrow M$ be a retraction map as in Definition~\ref{def-RetractMap}. For any $\theta\in [0,1]$ the map $R_d\colon TM \rightarrow M\times M$ given by $$R_d(x,v)=\left(R(x,-\theta v),R(x,(1-\theta)v) \right)$$ is a discretization map on $M$.		
		\end{proposition}
	
	\begin{proof} From the definition of retraction map it is immediate that $R_d(x,0)=(R(x,0),R(x,0))=(x,x)$ and $T_{0_x}R^2_x-T_{0_x}R^1_x=(1-\theta)\, T_{0_x}R_x+ \theta \, T_{0_x}R_x\equiv {\rm Id}_{T_xM}$.		
	\end{proof}

Starting from a retraction map we may define different discretization maps, as shown in the above proposition. In the sequel, we will see that these different maps will lead to known numerical methods. For step size $h$ and retraction map $R^h(x,v)=x+h\,v$ on the Euclidean space, one possible discretization  map is $R_d^h(x,v)=(x,x+h\, v)$ that corresponds with a first order integrator method as described, for instance, in~\cite{2006McLaPerl}. However, other discretization maps may be defined from the same retraction map to construct different integrators. For example, for step size $h$ and the above retraction map $R^h$ we define:
$$R_d^h(x,v)=(R^{-h/2}(x,v), R^{h/2}(x,v))\, ,$$ 
that corresponds with a second order method as described in~\cite{2006McLaPerl}.

Let us describe another method to generate more discretization maps from a given one that will be useful in Section~\ref{Sec_higherorder} to obtain some higher-order numerical methods.
Define the inversion map $I_M: M\times M\rightarrow M\times M$ by $I_M (x,y)=(y, x)$ for all $x, y\in M$.

\begin{proposition}\label{Prop:inversion}
If $R_d:U\subset TM\rightarrow M\times M$ is a discretization map, then 
$R^*_d: \overline{U}\subset TM\rightarrow M\times M$ with $\overline{U}=\{ (x, v)\in TM\; \mid \, (x,-v)\in U\}$ 
and defined by 
\[
R^*_d(x, v)=\left(I_M \circ R_d\right)(x, -v)
\]
is also a discretization map. The map $R^*_d$ is called the adjoint discretization map of $R_d$. 
\end{proposition}
\begin{proof}
Using the notation $$R_d^*(q, v)=((R^*)^1(x,v), (R^*)^2(x,v))=((R^*)^1_x(v), (R^*)^2_x(v))\, , $$ the two properties in Definition~\ref{def:DiscreteMap2} are satisfied by $R^*_d$:  
\begin{enumerate}
\item $R^*_d(x, 0)=\left(I_M\circ R_d\right)(x, 0)=I_M(x, x)=(x,x).$
\item $T_{0_x}(R^*)_x^2-T_{0_x}(R^*)_x^1$ is the identity map on $T_xM$ because $T_{0_x}(R^*)_x^2=-T_{0_x}R^1_x$ and $T_{0_x}(R^*)_x^1=-T_{0_x}R^2_x$.
\end{enumerate}
\end{proof}

\begin{definition}\label{symmetric-discretization}
A discretization map is {\bf symmetric} if $R^*_d=R_d$.
\end{definition}

	\begin{example}\label{Ex:basic}
		Let us provide some examples of retraction maps typically used in the literature for the construction of numerical methods, see~\cite{IserlesBook}, that can be used to define discretization maps satisfying the properties in Definition~\ref{def:DiscreteMap2}.
		
		\begin{enumerate}
			\item The explicit Euler method: $R_d(x,v)=(x,x+v)$ being its adjoint  $R^*_d(x,v)=(x-v,x)$ .
			\item The implicit midpoint rule is a symmetric discretization map: $R_d(x,v)=\left(x-\dfrac{v}{2},x+\dfrac{v}{2}\right)$.
			\item The $\theta$-method: $R_d(x,v)=\left(x-\theta\, v,x+(1-\theta)\, v\right)$ where $\theta\in [0,1]$.
		\end{enumerate}	
	As known, for $\theta\in \{0, 1/2\}$, we recover the first two maps from the third one. All these methods are defined on the Euclidean vector space ${\mathbb R}^n$. 
 \demo
	\end{example}
\begin{example}\label{Ex:Sphere}
		Given a Riemannian manifold $(M,g)$ and the associated exponential map $exp_x: T_xM\rightarrow M$
		we can define the following map 
		\begin{equation}\label{eq:Rdexp}
		R_d(x,\xi)=\left(exp_x(-\xi/2), exp_x(\xi/2)\right)\, ,
		\end{equation}
		that satisfies the properties in Definition~\ref{def:DiscreteMap2}  and it is a symmetric discretization map. Let us give some specific examples of discretization maps that can be associated with the exponential map. 
		
		For instance, on the sphere $S^2$ with the Riemannian metric induced by the restriction of the standard metric on ${\mathbb R }^3$ we have that 
		\[
		exp_x(\xi)=\cos ( \|\xi\|) \, x+\sin  (\|\xi\|) \,  \frac{\xi}{\|\xi\|}, \qquad \xi\in T_x S^2\, .
		\]
		Thus we move along the greatest circle that are the geodesics on the sphere. Remember that $	exp_x(0_x)=x$ and the exponential map is a continuous map. Hence, we can define the following discretization map on $M$:
		\begin{equation}\label{eq:RdS2mid}
		R_d (x, \xi)=\left(\cos \left( \frac{\|\xi\|}{2}\right) x-\sin  \left(\frac{\|\xi\|}{2}\right)  \frac{ \xi}{\|\xi\|}, \cos \left( \frac{\|\xi\|}{2}\right) x+\sin \left(\frac{\|\xi\|}{2}\right) \frac{ \xi}{\|\xi\|}\right)\, .
		\end{equation}
		Another option is to use as a retraction map on the sphere the projection 
		$R_x(\xi)=  \frac{x+\xi}{\|x+\xi\|}$ that leads to the following discretization map: 
		\[
			R_d (x, \xi)=\left(\frac{x-\xi/2}{\|x-\xi/2\|}, \frac{x+\xi/2}{\|x+\xi/2\|}\right).
		\] 
		Proposition~\ref{prop:buildExtRetraction} for $\theta=1/2$ guarantees that both maps are discretization maps.  \demo
	\end{example}
	
	\begin{example}\label{Ex:cayley}
	Consider a Lie group $G$ and denote by ${\mathfrak g}$ its Lie algebra. It is a fact that any element $\xi$ in the Lie algebra is in one-to-one correspondence with a left-invariant vector field on $G$, $X_{\xi}=T_eL_g(\xi)$, where $e$ is the identity element of $G$ and $L_g: G\rightarrow G$ denotes the left-translation map.
	If $\gamma_{\xi}: {\mathbb R}\rightarrow G$ is an integral curve of $X_{\xi}$ with initial condition $\gamma_{\xi}(0)=e$, then we can generate a map between the Lie algebra and the Lie group called the exponential map: $\textrm{exp}(\xi)=\gamma_{\xi}(1)$.
	It is possible to check that the map $R: TG\rightarrow G$ given by
	\[
	(g, X)\longrightarrow g\, \textrm{exp}(T_g L_{g^{-1}}(X))
	\]
	is a retraction map where $X(g)\in T_gG$. Then, we define a symmetric discretization  map on the Lie group $G$, $R_d: TG\rightarrow G\times G$, as follows
	\[
	R_d (g, X)=\left(g\, \textrm{exp}\left(- \frac{1}{2}T_g L_{g^{-1}}(X)\right),  \, g\,\textrm{exp}\left(\frac{1}{2}T_g L_{g^{-1}}(X)\right)\right)\,.
	\]
The properties in Definitions~\ref{def-RetractMap} and~\ref{def:DiscreteMap2}	are satisfied because the tangent map $T_e\exp$ is the identity map. 
	
In the case of $SO(3)=\{ A \in GL(3, {\mathbb R})\; \mid \, AA^T=A^TA={\rm Id}_3,\  \det A=1\}$ 
	we have that 
	an element $(A, X)\in TSO(3)$  is given by a pair of matrices such that $A\in SO(3)$ and 
	$XA^T+AX^T=0$. Therefore, the Lie algebra ${\mathfrak so}(3)$ is the set of skew-symmetric matrices: 
	$\xi=A^TX\in {\mathfrak so}(3)$. The above retraction map for $SO(3)$ becomes: \[
	R(A,X)=A\,\textrm{exp} (A^TX) \,.
	\]
	The exponential map could be replaced by the Cayley transformation: 
	\[\begin{array}{rcl} \textrm{cay}: \mathfrak{so}(3) &\longrightarrow & SO(3)\\ \xi &\longmapsto &
 \textrm{cay} (\xi)=({\rm Id}_3-\xi/2)^{-1}({\rm Id}_3+\xi/2)\, , \end{array}
	\]
	where ${\rm Id}_3$ stands for the identity matrix. Then we define the following retraction map $R_{\rm cay}\colon TSO(3)\rightarrow SO(3)$: 
		\begin{equation}\label{eq:Rcaymid}
	R_{\rm cay}(A,X)= A\,\textrm{cay} (A^TX)=A({\rm Id}_3-A^TX/2)^{-1}({\rm Id}_3+A^TX/2)\,.
	\end{equation}
Using Proposition~\ref{prop:buildExtRetraction}, we obtain the following discretization map 

\noindent $R_{d,{\rm cay}}\colon TSO(3)\rightarrow SO(3)\times SO(3)$:
		\begin{align*}
	&R_{d,{\rm cay}}(A,X)= \left(R_{\rm cay}(A,-X/2), R_{\rm cay}(A,X/2)\right)\\ &= \left(A({\rm Id}_3+A^TX/4)^{-1}({\rm Id}_3-A^TX/4), A({\rm Id}_3-A^TX/4)^{-1}({\rm Id}_3+A^TX/4)\right)\,. \demo
	\end{align*} 
	\end{example}

\section{Lift of discretization maps} \label{Sec:liftRd}
	
	We can construct discretization maps, as described in Definition~\ref{def:DiscreteMap2}, on any manifold. When studying mechanical systems, it may be useful to define discretization maps on the tangent bundle for the Lagrangian framework or on the cotangent bundle for the Hamiltonian framework. As discretization maps can be defined on different manifolds, we introduce the notation $R_d^{TM}$ so that the superscript tells us the domain of such a map. Thus, the map $R_d^{TM}\colon TM \rightarrow M\times M$ is called a discretization map on $M$.  Note that ``on $M$" emphasizes where the image takes values. The manifold $M$ could be equal to the tangent bundle $TQ$ or to the cotangent bundle $T^*Q$ depending on the dynamics under study.
	
	Here, we are interested in constructing specific discretization maps on the tangent and cotangent bundles obtained from discretization maps on the base manifold. The objective is to generate geometric integrators for mechanical systems by using a suitable  notion of lifted  discretization maps to the tangent and cotangent bundles to encompass both the Lagrangian and the Hamiltonian framework.

	We first review the notion of tangent and cotangent lift of a map between manifolds, see~\cite{MR_book}. 
	
		Let $M_1$ and $M_2$ be $n$-dimensional manifolds and $F: M_1\rightarrow M_2$ be a smooth map. The
	{\bf tangent lift} $TF: TM_1\rightarrow TM_2$  of $F$ is defined by 
		\[
		TF(v_x)=T_xF (v_x)\in T_{F(x)} M_2\, , \qquad \mbox{ where } v_x\in T_xM_1\; ,
		\] and $T_xF$ is the tangent map of $F$ whose matrix is the Jacobian matrix of $F$ at $x\in M_1$ in a local chart.

		As the tangent map $T_xF$ is linear, the dual map $T_{x}^*F\colon T^*_{F(x)}M_2\rightarrow T^*_xM_1$ is defined as follows:
  \[\langle(T^*_{x}F)(\alpha_2), v_{x}\rangle=\langle \alpha_2, T_{x}F(v_{x})\rangle\mbox{ for every } v_x\in T_xM_1.\]
  Note that $(T^*_{x}F)(\alpha_2)\in T^*_xM_1$. 
  
 To define the cotangent lift in Section~\ref{Sec:coTliftRd}, we need the cotangent lift of the inverse of the discretization map. Thus, we fix the notation for such a cotangent lift. 
	
	\begin{definition}\label{def:colift}
Let $F: M_1\rightarrow M_2$ be a diffeomorphism. The vector bundle morphism 
	$\widehat{F}: T^*M_1\rightarrow T^*M_2$ defined by
	\[
	\widehat{F}=T^*F^{-1}
	\]
	is called the cotangent lift of $F^{-1}$. 
	 \end{definition}
In other words, $\widehat{F}(\alpha_x)= 	T^*_{F(x)}F^{-1} (\alpha_x)$ where $\alpha_x\in T^*_x M_1$. Obviously, $(T^*F^{-1})\circ (T^*F)={\rm Id}_{T^*M_2}$.
	
We quickly review here some notions from symplectic geometry, see~\cite{LiMarle}. Denote by $\pi_M: T^*M\rightarrow M$ the canonical projection of the cotangent bundle and define the Liouville 1-form $\theta_M$ on $T^*M$ by 
$\langle \theta_M(\alpha_x), X_{\alpha_x}\rangle = \langle \alpha_x, T_{\alpha_x}\pi_M( X_{\alpha_x})\rangle$ where $X_{\alpha_x}\in T_{\alpha_x} T^*M$ and denote by $\omega_{M}=-d\Theta_M$ the canonical symplectic 2-form on $T^*M$. 
Thus $(T^*M,\omega_M)$ is a symplectic manifold. For a diffeomorphism $F: M_1\rightarrow M_2$,  we recall the well-known proposition for symplectic manifolds in~\cite{LiMarle}. 
	
	\begin{proposition} \label{Prop_CotLift} Let $F: M_1\rightarrow M_2$ be a diffeomorphism. The cotangent lift  $ \widehat{F}: T^*M_1\rightarrow T^*M_2$ of $F^{-1}$ is a symplectomorphism for the symplectic manifolds $(T^*M_1, \omega_{M_1})$ and $(T^*M_2, \omega_{M_2})$. In other words, the symplectic 2-form is preserved by the pull-back of $\hat{F}$: 
		\[
		 \widehat{F}^*(\omega_{M_2})=\omega_{M_1}\, \mbox{ where } \widehat{F}^*\colon \Omega^2(T^*M_2) \rightarrow \Omega^2(T^*M_1)\, .
		\]
		Equivalently, the inverse of the cotangent lift $\widehat{F}^{-1}\colon T^*M_2\rightarrow T^*M_1$ is also a symplectomorphism.
	\end{proposition}
Some expressions in coordinates will be useful in the sequel. Take   local coordinates $q=(q^1,\ldots , q^n)$ on $M_1$ and $x=(x^1,\ldots , x^m)$ on $M_2$ and induced coordinates $(q, v)$ on $TM_1$ and $(x, u)$ on $TM_2$, respectively.  If $F: M_1\rightarrow M_2$ is written in local coordinates as $(q^1, \ldots, q^n)\rightarrow  (F^1(q), \ldots, F^m( q))$
Then 
\begin{eqnarray*}
	TF({q}, {v})&=&\left(F^i( q)\; ;\; \frac{\partial F^i}{\partial q^j}(q)v^j\right)\; .
\end{eqnarray*}
	Taking now induced coordinates $(q, p)$ on $T^*M_1$ and $(x, r)$ on $T^*M_2$ we have  
	\begin{eqnarray*}
		\widehat{F}(q, p)
		&=&\left(F^i( q)\; ;\;
		 p_j\frac{\partial (F^{-1})^j}{\partial q^i}
		 (F( q))\right)\; .
	\end{eqnarray*}
We could also use the matrix notation:
$$D_{q}F=\left( \frac{\partial F^i}{\partial q^j}(q)\right)_{1\leq i, j\leq \dim M_1} \quad\mbox{ and }\quad 
D_{F(q)}F^{-1}=\left(\frac{\partial (F^{-1})^i}{\partial x^j}(F( q))\right)_{1\leq i, j\leq \dim M_2}\,.$$ 
Note that $$D_{F( q)}F^{-1}=\left[D_{ q}F\right]^{-1}\; .$$
When we restrict the previous maps $TF$ and $\widehat{F}$ to a fiber we induce the maps
\[
\begin{array}{rrcl}
	T_{q}F:& T_{ q} M_1&\rightarrow& T_{F( q)}M_2\\
      & { v}&\longmapsto& D_{ q}F\; {v}^T
\end{array}
\]
and
\[
\begin{array}{rrcl}
\widehat{F}_{ q}:& T^*_{ q} M_1&\rightarrow& T^*_{F(q)}M_2\\
& { p}&\longmapsto& ((D_{ q} F)^{-1})^T  { p}^T= \left({ p} (D_{ q} F)^{-1}\right)^T\, .
\end{array}
\]
Consequently,
\begin{equation}\label{eq:hatFinverse}
\begin{array}{rrcl}
\widehat{F}^{-1}_{F(q)}:& T^*_{F( q)} M_2&\rightarrow& T^*_{ q}M_1\\
& {r}&\longmapsto& { r} D_{ q} F\, .
\end{array}
\end{equation}

	\subsection{Tangent lift of discretization maps} \label{Sec:TliftRd}
	
	We prove that if we suitably lift the discretization map $R_d\colon TQ \rightarrow Q\times Q$ on $Q$ in Definition~\ref{def:DiscreteMap2}, we obtain a new discretization map on the tangent bundle $TQ$. These constructions are able to provide a geometric framework to obtain numerical integrators for second-order differential equations (SODEs), see Section~\ref{Sec:AllSODE},  and for the dynamics of mechanical systems as shown in Sections~\ref{Sec:GeomInt} and~\ref{Sec:Compose}.
	
	Remember that the notation $R_d^{TTQ}$ for a discretization map on $TQ$ makes clear the manifold to be discretized, that is, $R_d^{TTQ}\colon TTQ \rightarrow TQ\times TQ$.  To define it from a discretization map $R_d:{TQ}\rightarrow Q\times Q$ on $Q$ is necessary to use the canonical involution map $\kappa_Q$ that shows the double vector bundle structure of the vector bundle $TTQ$ and defines a vector bundle isomorphism, as described for instance in~\cite{Tu,1999TuUr}. 
	
	Let us recall here the definition of the canonical involution. 
	Let $Q$ be a smooth manifold of dimension $n$, $\tau_{Q}:TQ\rightarrow Q$ be the canonical tangent bundle projection and $TTQ$ the double tangent bundle of $Q$. The manifold $TTQ$ naturally admits two vector bundle structures.
	The first vector bundle structure is the canonical one with vector bundle projection $\tau_{TQ}: TTQ\rightarrow TQ$. For the second vector bundle structure of $TTQ$, the vector bundle projection is given by the tangent map $T\tau_{Q}:TTQ\rightarrow TQ$.
	The canonical involution $\kappa_{Q}: TTQ\rightarrow TTQ$ is a vector bundle isomorphism (over the identity of $TQ$) between the two previous vector bundles. In fact, $\kappa_{Q}$ is characterized by the following condition: let $\Phi:U\subseteq {\mathbb R}^{2}\rightarrow Q$ be a smooth map on an open subset $U$ of ${\mathbb R}^{2}$ defined by 
	\[
		(t,s)\mapsto \Phi(t,s)\in Q,
	\]
	then
	\begin{equation*}
	\kappa_{Q}\left( \frac{\partial }{\partial t}\frac{\partial}{\partial s} \Phi(t,s) \right)= \frac{\partial}{\partial s}\frac{\partial}{\partial t} \Phi(t,s).
	\end{equation*}
	Note that $\kappa_{Q}$ is an involution of $TTQ$, that is, $\kappa_{Q}^{2}={\rm Id}_{TTQ}$.
	If $(q,v)$ are canonical fibered coordinates of $TQ$ and $(q,v,\dot{q},\dot{v})$ are the corresponding local fibered coordinates of $TTQ$, then
	\begin{equation*}\label{def:local:can:invol}
	\kappa_{Q}(q,v, \dot{q},\dot{v})=(q,\dot{q}, v, \dot{v}).
	\end{equation*}
Having all this in mind, remember that the tangent lift of a vector field $X$ on $Q$ does not define a vector field on $TQ$. It is necessary to consider the composition $\kappa_Q\circ TX$ to obtain a vector field on $TQ$ that is called complete lift $X^c$ of  the vector field $X$. A similar trick must be used to lift a discretization map  from $TQ$ to $TTQ$ as shown in the following diagram.

	\begin{equation*}
	\xymatrix{TTQ \ar[rr]^{R^{T}_d} \ar[d]<2pt>^{\kappa_{Q}} && TQ\times TQ \\
	 TTQ \ar[u]<2pt>^{\kappa_{Q}} \ar[d]^{\tau_{TQ}} \ar[rr]^{{\rm T}R_d} && T(Q\times Q) \ar@{=}[u] \ar[d]^{\tau_{Q\times Q} } \\ TQ \ar[rr]^{R_d} && Q\times Q }
	\end{equation*}
	
Note that $T(Q\times Q)$ and $TQ\times TQ$ are trivially identified since any vector on $T(Q\times Q)$ is given as a tangent vector at 0 of a curve $\sigma: {\mathbb R}\rightarrow Q\times Q$, that is, $\dot{\sigma}(0)\in T_{\sigma(0)}(Q\times Q)$. As $\sigma$ has two components $\sigma(t)=(\sigma_1 (t), \sigma_2(t))$ where $\sigma_i: {\mathbb R}\rightarrow Q$, $i=1,2$, the identification  
$\dot{\sigma}(0)\equiv (\dot{\sigma}_1(0),\dot{\sigma}_2(0))\in T_{\sigma_1(0)}Q\times T_{\sigma_2(0)}Q$ is made. The following proposition shows that ${\rm T}R_d\circ \kappa_Q$ is a discretization map on $TQ$.  From now on, such a map is denoted by $R_d^T$ to emphasize it is obtained by tangently lifting $R_d$.
	\begin{proposition}\label{prop-liftTR} If $R_d$ is a discretization map on $Q$, then $R_d^{T}={\rm T}R_d\circ \kappa_Q$ is a discretization map on $TQ$.
	\end{proposition}

\begin{proof} In local coordinates $(q,v,\dot{q},\dot{v})$ of $TTQ$ we have that ${\rm T}R_d(q,v,\dot{q},\dot{v})=(R_d(q,v),{\rm D}_{(q,v)} R_d(q,v) \, (\dot{q},\dot{v})^T)$ and $$R_d^{T}(q,\dot{q},v,\dot{v})=(R_d(q,v),{\rm D}_{(q,v)} R_d \; (\dot{q},\dot{v})^T)\; .$$ Remember the abuse of notation because $T(Q\times Q)$ and $TQ \times TQ$ are trivially identified.

Let us prove that the properties in Definition~\ref{def:DiscreteMap2} are satisfied by $R_d$ knowing that $R_d(q,v)=(R^1(q,v),R^2(q,v))$.
	\begin{enumerate}
		\item 
		We know that  $R_d(q,0)=(q,q)$ for all $q\in Q$. Consequently, 
		\begin{equation*} R_d^{T}(q,\dot{q},0,0)=(R_d(q,0),{\rm D}_{(q,0)} R_d \, (\dot{q},0)^T)=(q,q,\dot{q},\dot{q})\equiv (q,\dot{q}; q,\dot{q})\, , \end{equation*} 
		where  we use the natural  identification between $T(Q\times Q)$ and $TQ\times TQ$.

	\item For the second property, we know that 
	$$R_d^{T}(q,\dot{q},v,\dot{v})=((TR^1)(q,v; \dot{q},\dot{v}), (TR^2)(q, v; \dot{q},\dot{v}))\, .$$
We need to compute
$${\rm T}_{(0,0)_{(q,\dot{q})}}({\rm T}R^a)_{(q, \dot{q})}: T_{(0,0)_{(q,\dot{q})}}T_{(q, \dot{q})}TQ\equiv T_{(q, \dot{q})}TQ\rightarrow T_{(q, \dot{q})}TQ\; ,
$$for $a=1,2$,
 to prove that the map ${\rm T}_{(0,0)_{(q,\dot{q})}}({\rm T}R^2)_{(q, \dot{q})}-{\rm T}_{(0,0)_{(q,\dot{q})}}({\rm T}R^1)_{(q,\dot{q})}$ is the identity map understood as an application from $T_{(q,\dot{q})}TQ$ to itself. 
	
	At $(q,\dot{q},0,0)$, the   linear map ${\rm T}_{(0,0)_{(q,\dot{q})}}({\rm T}R^a)_{(q, \dot{q})}$ is given by the following matrix
	$$\begin{pmatrix}
\partial_{v^j}(R^a)^i(q, 0) & 0\\  \partial_{q^k} \partial_{v^j}(R^a)^i(q,0)\dot{q}^k & \partial_{v^j}(R^a)^i(q,0) 
	\end{pmatrix}\, ,$$
after calculating 
\[\left.
\frac{d}{dt}\right\vert_{t=0} \left(
R_d^a(q, tv), 
\partial_{q^j} R_d^a (q, tv)\dot{q}^j+\partial_{v^j} R_d^a (q, tv)t\dot{v}^j\right)\: .
\]
	Using again the properties of the discretization map  $R_d$,  
	the Jacobian matrix of $({\rm T}R^2)_{(q, \dot{q})}-({\rm T}R^1)_{(q,\dot{q})}$ at $(0,0)_{(q,\dot{q})}$ is: 
	$$\begin{pmatrix}
	\partial_vR^2(q, 0)- \partial_vR^1(q, 0)& 0\\  \partial_q(\partial_vR^2(q,0)-\partial_vR^1(q,0))\dot{q} & 	\partial_vR^2(q, 0)- \partial_vR^1(q, 0)
	\end{pmatrix}= {\rm Id}_{2n\times 2n}\, ,$$ 
	as needed. Note that  $\partial_q\partial_v(R^2-R^1)(q,0)=0$ because $\partial_v(R^2-R^1)(q,0)={\rm Id}_{n\times n}$.
	\end{enumerate}
	\end{proof}
\begin{remark}
If we use the discretization map obtained from the exponential map of a Riemannian metric $g$ as in Equation~\eqref{eq:Rdexp}, then the tangent lift of this specific discretization map is associated with the complete lift of $g$, denoted by $g^C$, which is a semi-riemannian metric on $TQ$ (see details in \cite{YaIs73,AMM2}). \demo

\end{remark}

\begin{proposition}\label{lift-adjoint}
Let $R_d$ be a discretization map and $R_d^*$ be the adjoint discretization map. Then the tangent lift of a symmetric discretization map is also symmetric, that is, 
$
(R_d^{*})^T= (R_d^T)^{*}$. 
\end{proposition}
\begin{proof}
It is simple to check that
\begin{eqnarray*}
(R_d^*)^T(q, \dot{q}, v, \dot{v})&=&
TR_d^*(q, v, \dot{q}, \dot{v})
=(T(R^*)^1(q, v, \dot{q}, \dot{v}), T(R^*)^2(q, v, \dot{q}, \dot{v}))\\
&=&(TR^2(q, v, -\dot{q}, -\dot{v}), TR^1(q, v, -\dot{q}, -\dot{v}))\\
&=&I_{TQ}(TR^1(q, v, -\dot{q}, -\dot{v}), TR^2(q, v, -\dot{q}, -\dot{v}))\\
&=&(R_d^T)^*(q, \dot{q}, v, \dot{v})
\end{eqnarray*}
where $I_{TQ}(u_q, v_q)=(v_q, u_q)$ for all  $u_q, v_q\in T_qQ$. 
\end{proof}

	\subsection{Cotangent lift of discretization maps} \label{Sec:coTliftRd}

To encompass the Lagrangian and Hamiltonian dynamics together to build numerical integrators, we are interested in defining a very particular notion of discretization map on the cotangent bundle.

Given a discretization map $R_d: TQ\rightarrow Q\times Q$ we know that the cotangent lift
$\widehat{R_d}: T^*TQ\rightarrow T^*(Q\times Q)$ is a symplectomorphism between the symplectic manifolds $(T^*TQ, \omega_{TQ})$ and $(T^*(Q\times Q), \omega_{Q\times Q})$ as mentioned in Proposition~\ref{Prop_CotLift}.

 	According to Definition~\ref{def:colift}, in local coordinates $(q,v,p_q,p_v)$ for $T^*TQ$ the cotangent lift of $R_d$ is given by:
 	\begin{equation*}
 	\begin{array}{rrl}
 		\widehat{R_d}\colon& T^*TQ \longrightarrow &T^*(Q\times Q) \\
& 		(q,v,p_q,p_v)  \longmapsto & \left(R_d(q,v), \left( p_q, \; p_v\right) \, (D_{(q,v)}R_d)^{-1}\right)
 		\end{array}
 	\end{equation*}
 	where $(D_{(q,v)}R_d)^{-1}$ is the inverse of the Jacobian matrix of $R_d$.

We use the cotangent lift $\widehat{R_d}$ of the discretization map on $Q$ to define a discretization map on $T^*Q$ that must be a map from $TT^*Q$ to $T^*Q\times T^*Q$.

For this purpose it is necessary to use the canonical symplectomorphism $\alpha_Q: TT^*Q\rightarrow T^*TQ$ between double vector bundles (see \cite{TuHamilton,1999TuUr}).
Locally,
\begin{equation*}
\begin{array}{crcl} \alpha_Q\colon & TT^*Q & \longrightarrow &  T^*TQ\\&
(q,p,\dot{q},\dot{p}) & \longrightarrow & (q,\dot{q},\dot{p},p).
\end{array}
\end{equation*} As described in~\cite{TuHamilton}, the symplectomorphism $\alpha_Q$ is between the sympletic manifold $(TT^*Q, {\rm d}_T\omega_Q)$ and the natural symplectic manifold $(T^*TQ, \omega_{TQ})$. Recall that in local coordinates $(q,p,\dot{q},\dot{p})$ for $TT^*Q$, the symplectic form $d_T\omega_Q$ has the following expression: ${\rm d}_T \omega_Q= {\rm d}q\wedge {\rm d}\dot{p}+{\rm d}\dot{q}\wedge {\rm d}p$. Moreover, we need the diffeomorphism 
\[
\begin{array}{rrcl}
\Phi:&T^*Q\times T^*Q &\longrightarrow &T^*(Q\times Q)\\
& (q_0, p_0; q_1, p_1)&\longmapsto& (q_0, q_1, -p_0, p_1)
\end{array}
\]
which is also a  symplectomorphism between $(T^*(Q\times Q), \omega_{Q\times Q})$     and   
$(T^*Q\times T^*Q, \Omega_{12}=pr_2^*\omega_Q-pr^*_1\omega_Q)$, where ${\rm pr}_i\colon T^*(Q\times Q) \rightarrow T^*Q\times T^*Q$ denotes the projection into the $i$--th factor of the cartesian product in the image. 

The following diagram shows how to define the discretization map on $T^*Q$ from the one on $Q$.

	\begin{equation*}
\xymatrix{ TT^*Q  \ar[rr]^{R_d^{T^*}}\ar[d]_{\alpha_{Q}} && T^*Q\times T^*Q   \\ T^*TQ \ar[d]^{\pi_{TQ}}\ar[rr]^{	\widehat{R_d}}&& T^*(Q\times Q)\ar[u]_{\Phi^{-1}}\ar[d]^{\pi_{Q\times Q}}\\ TQ \ar[rr]^{R_d} && Q\times Q }
\end{equation*}

Now we prove that $R_d^{T^*}$ is a discretization map on $T^*Q$ according to Definition~\ref{def:DiscreteMap2}. From now on, it will be called the cotangent lift of $R_d$.

\begin{proposition} \label{prop_liftT^*Rd}
	Let $R_d\colon TQ\rightarrow Q\times Q$ be a discretization map on $Q$ as in Definition~\ref{def:DiscreteMap2}. Then $R_d^{T^*}=\Phi^{-1}\circ \widehat{R_d}\circ \alpha_Q\colon TT^*Q\rightarrow T^*Q\times T^*Q$ is a discretization map on  $T^*Q$.
\end{proposition}

\begin{proof}
	 Let us compute the cotangent lift  of the tangent map of $R_d$ for local coordinates $(q,v,p_q,p_v)$ of $T^*TQ$:
	\begin{equation*}
	\widehat{R_d}(q,v,p_q,p_v)= \left(R_d(q,v), (p_q,\; p_v) ({\rm D}_{(q,v)}R_d)^{-1}\right)\,.
	\end{equation*}
Expressing  the inverse of $({\rm D}R_d)^{-1}$ as a matrix with two blocks  $({\rm D}R_d)_i^{-1}$ of size $2n\times n$, $i=1,2$, that is 
\[
({\rm D}_{(q,v)}R_d)^{-1}= \begin{pmatrix}
({\rm D}_{(q,v)}R_d)_1^{-1}&
	({\rm D}_{(q,v)}R_d)_2^{-1}
	\end{pmatrix}
\]
and 
$$({\rm D}_{(q,v)}R_d)^{-1}=\begin{pmatrix}
\partial_q R^1(q,v)&\partial_v R^1(q,v) \\
\partial_q R^2(q,v)&\partial_v R^2(q,v)
\end{pmatrix}^{-1}\, .
$$
We can write 
		\begin{equation*}
	R_d^{T^*}(q,p,\dot{q},\dot{p})= \left(R^1(q,\dot{q}),- (\dot{p}, \; p) ({\rm D}_{(q,\dot{q})}R_d)_1^{-1}; R^2(q,\dot{q}),(\dot{p}, \; p) ({\rm D}_{(q,\dot{q})}R_d)_2^{-1}\right)\,.
	\end{equation*}
	Let us check if it satisfies the properties in Definition~\ref{def:DiscreteMap2}:
	\begin{enumerate}
		\item  
		Note that  the Jacobian matrix of $R_d$ at $(q,0)$ is 
		$${\rm D}_{(q,0)} R_d=\begin{pmatrix} {\rm Id} & \partial_vR^1(q,0)\\{\rm Id} & \partial_vR^2(q,0)\end{pmatrix}\, . $$
As $\partial_vR^2(q,0)-\partial_vR^1(q,0)={\rm Id}$, the inverse is
	\begin{equation*}({\rm D}_{(q,0)} R_d)^{-1}=\begin{pmatrix} {\rm Id}+\partial_vR^1(q,0)& -\partial_vR^1(q,0)\\-{\rm Id}&{\rm Id}\end{pmatrix}\, . \end{equation*}
Thus,
	\begin{equation*}
	R_d^{T^*}(q,p,0,0)= (R^1(q,0),-(0, \; p) ({\rm D}_{(q,0)}R_d)_1^{-1}; R^2(q,0),(0, \; p) ({\rm D}_{(q,0)}R_d)_2^{-1})\,,
\end{equation*}
and it is straightforward that 	$R_d^{T^*}(q,p,0,0)=(q,p; q,p)$.

		\item We must prove that $T_{(q,p,0,0)}\left(R^{T^*}_d\right)^2_{(q,p)}-T_{(q,p,0,0)}\left(R^{T^*}_d\right)_{(q,p)}^1$ is the identity map from $T_{(q,p,0,0)}TT^*Q\simeq T_{(q,p)}T^*Q$ to itself. 
	
Let us compute the following derivatives for $a=1,2$:
$$
\left.\frac{d}{dt}\right\vert_{t=0}\left(R^{T^*}_d\right)^a(q,p, t\dot{q}, t\dot{p})\, .
$$

For instance, for $i=1$ we have 
\[
\left.\frac{d}{dt}\right\vert_{t=0}\left(R^{T^*}_d\right)^1(q,p, t\dot{q}, t\dot{p})=
\left.\frac{d}{dt}\right\vert_{t=0}\left[R^1(q,t\dot{q}),- (t\dot{p}, \; p) ({\rm D}_{(q,t\dot{q})}R_d)_1^{-1}\right]\; .
\]
Using the expression for the derivative of an inverse matrix, we have that $\left.\frac{d}{dt}\right\vert_{t=0}({\rm D}_{(q, t\dot{q})}R_d)^{-1}$ is equal to
\begin{eqnarray*}
&&-\begin{pmatrix} {\rm Id}+A& -A\\-{\rm Id}&{\rm Id}\end{pmatrix}
\begin{pmatrix}
\partial_{v}\partial_qR^1(q,0)&\partial_{v}\partial_vR^1(q,0)\\
\partial_{v}\partial_qR^2(q,0)&\partial_{v}\partial_vR^2(q,0)
\end{pmatrix}
\begin{pmatrix} {\rm Id}+A& -A\\-{\rm Id}&{\rm Id}\end{pmatrix}\\
&=&\begin{pmatrix}  {\mathbf *}&{\mathbf *}\\\partial_{v}\partial_vR^1(q,0)-\partial_{v}\partial_vR^2(q,0)&\partial_{v}\partial_vR^2(q,0)-\partial_{v}\partial_vR^1(q,0)\end{pmatrix}
\end{eqnarray*}
where $A=\partial_vR^1(q,0)$ and $(*)$ denotes terms that are not explicitly needed in the computations. We have used that $\partial_q\partial_v(R^2-R^1)(q,0)=0$ since $\partial_v(R^2-R^1)(q,0)={\rm Id}_{n\times n}$. Thus, 
\begin{eqnarray*}
&&\left.\frac{d}{dt}\right\vert_{t=0}\left[R^1(q,t\dot{q}),- (t\dot{p}, \; p) ({ D}_{(q,t\dot{q})}R_d)_1^{-1}\right]\\
&&=
\begin{pmatrix}
\partial_vR^1(q,0)&0\\
p(\partial_{v}\partial_vR^1(q,0)-\partial_{v}\partial_vR^2(q,0))&-{\rm Id}-\left(\partial_v R^1(q,0)\right)^T
\end{pmatrix}
\begin{pmatrix}
\dot{q}\\
\dot{p}
\end{pmatrix}
\end{eqnarray*}

Analogously,
\begin{eqnarray*}
&&\left.\frac{d}{dt}\right\vert_{t=0}\left[R^2(q,t\dot{q}), (t\dot{p}, \; p) ({ D}_{(q,t\dot{q})}R_d)_2^{-1}\right]\\
&&=
\begin{pmatrix}
\partial_vR^2(q,0)&0\\
p(\partial_{v}\partial_vR^1(q,0)-\partial_{v}\partial_vR^2(q,0))&-\left(\partial_v R^1(q,0)\right)^T
\end{pmatrix}
\begin{pmatrix}
\dot{q}\\
\dot{p}
\end{pmatrix}
\end{eqnarray*}
As a result,  
\[
\begin{pmatrix}
\partial_vR^2(q,0)&0\\
C&-\partial_v R^1(q,0)
\end{pmatrix}
-\begin{pmatrix}
\partial_vR^1(q,0)&0\\
C&-{\rm Id}-\partial_v R^1(q,0)
\end{pmatrix}
=\begin{pmatrix}
{\rm Id}&0\\0&{\rm Id}
\end{pmatrix}
\]
where $C=p(\partial_{v}\partial_vR^1(q,0)-\partial_{v}\partial_vR^2(q,0))$.
\end{enumerate}
\end{proof}

As the composition of symplectomorphisms is  a symplectomorphism~\cite{LiMarle}, the following result is straightforward.
\begin{proposition}\label{prop:sympl}
	Let $R_d\colon TQ \rightarrow Q\times Q$ be a retraction map on $Q$, then $R_d^{T^*}=\Phi^{-1}\circ 	\widehat{R_d} \circ \alpha_Q\colon T(T^*Q)\rightarrow T^*Q\times T^*Q$ is a symplectomorphism between $(T(T^*Q), {\rm d}_T \omega_Q)$ and $(T^*Q\times T^*Q, \Omega_{12})$.
\end{proposition}

As a consequence, 
\[
\left(R_d^{T^*}\right)^*(\Omega_{12})=d_T\omega_Q\, .
\]
The above result is essential to obtain symplectic methods in the following sections. 

When constructing numerical integrators in Section~\ref{Sec:GeomInt} for Hamiltonian systems, the inverse map of $R_d^{T^*}: TT^*Q\rightarrow T^*Q\times T^*Q$ is useful. Using Proposition \ref{prop_liftT^*Rd} we specifically write the inverse map 
\begin{equation*}
\left(R_d^{T^*}\right)^{-1}=\alpha^{-1}_Q\circ \widehat{R_d}^{-1} \circ \Phi: T^*Q\times T^*Q\rightarrow TT^*Q \, .
\end{equation*}

In local coordinates $(q_0,p_0;q_1,p_1)$ for $T^*Q\times T^*Q$ and using~\eqref{eq:hatFinverse}, it is quite simple to compute the inverse map 
	\begin{equation}\label{eq:localinverseRd*}
	\left(R_d^{T^*}\right)^{-1}(q_0,p_0; q_1,p_1)= \alpha_Q^{-1}\left(R_d^{-1} (q_0,q_1), (-p_0, p_1)\, { D}_{R_d^{-1} (q_0,q_1)}R_d\right) \, .
	\end{equation}
Remember that $\alpha_Q^{-1}(q,v, p_q, p_v)=(q, p_v, v, p_q)$. 

\subsection{Duality between the cotangent and the tangent lift of discretization maps}\label{Sec:DualLift}

After introducing both the tangent and cotangent lift of discretization maps, we show here the existing duality between the two maps. 
	
	For a discretization map on $Q$, we consider the tangent lift $R_d^{T}: TTQ\rightarrow TQ\times TQ$ defined by 
	$R_d^{T}=TR_d\circ \kappa_Q$ and the corresponding cotangent lift 
	$R_d^{T^*}=\Phi^{-1}\circ \widehat{R_d}\circ \alpha_Q\colon TT^*Q\rightarrow T^*Q\times T^*Q$. 
As mentioned earlier, $TT^*Q$ is a symplectic manifold with the 2-form ${\rm d}_T \omega_Q$ that induces a natural pairing as follows. Let $v\in TT^*Q$ and let $w\in TTQ$ such that $\tau_{TQ}(w)=T\pi_Q(v)$, the pairing $\langle\cdot, \cdot \rangle^T$ induced by the symplectic structure of $TT^*Q$ is given by 
\[
\langle v, \kappa_Q(w)\rangle^T= \frac{d}{dt}\langle \sigma_v(t), \gamma_{\tilde{w}}(t)\rangle (0)=\langle \alpha_Q(v), w\rangle\, ,
\]
where $\alpha_Q: TT^*Q\rightarrow T^*TQ$, $\sigma_v: I\rightarrow T^*Q$ and $\gamma_{\tilde{w}}: I\rightarrow TQ$ satisfy $\dot{\sigma}_v(0)=v$ and $\dot{\gamma}_{\tilde{w}}(0)=\tilde{w}$ with $\tilde{w}=\kappa_Q(w)$ and $\pi_Q\circ \sigma_v=\tau_Q \circ \gamma_{\tilde{w}}$.

\begin{proposition}\label{Prop:RTandRT*} The tangent lift and the cotangent lift of a discretization map on $Q$ satisfy the following equality:
\[
\left\langle \Phi(\alpha_{q_0}, \alpha_{q_1}),  R_d^{T} (w)\right\rangle
=\left\langle 	\left(R_d^{T^*}\right)^{-1}(\alpha_{q_0}, \alpha_{q_1}), w\right\rangle^T \, ,
\]	
	where $w\in TTQ$, $(R_d)^{-1}(q_0, q_1)=T\tau_Q(w)$ and the pairing $\langle\cdot, \cdot \rangle^T$ is induced by the symplectic structure of $TT^*Q$.
\end{proposition}	
\begin{proof}
	Observe that 
	\begin{eqnarray*}
	\left\langle   \Phi(\alpha_{q_0}, \alpha_{q_1}), R_d^{T} (w)\right\rangle
	&=&\left\langle  (\widehat{R_d}^{-1}\circ \Phi)(\alpha_{q_0}, \alpha_{q_1}), \kappa_Q(w)\right\rangle \\	
	&=&\left\langle  (\alpha_Q)^{-1}\left( \left(\widehat{R_d}^{-1}\circ \Phi\right)(\alpha_{q_0}, \alpha_{q_1})\right), w \right\rangle^T\\
	&=&\left\langle 	\left(R_d^{T^*}\right)^{-1}(\alpha_{q_0}, \alpha_{q_1}), w\right\rangle^T \, .
	\end{eqnarray*}
\end{proof}	

Using Propositions \ref{lift-adjoint} and \ref{Prop:RTandRT*} it is easy to prove the following relation between the cotagent lift of the adjoint discretization and the adjoint of the cotangent lift discretization map. 
\begin{proposition} \label{prop_cotangent_sym} Let $R_d$ be a discretization map and $R_d^*$ be the adjoint discretization map. Then the cotangent lift of a symmetric discretization map is also symmetric, that is, 
$
(R_d^*)^{T^*}= (R_d^{T^*})^*$. 
\end{proposition}

\subsection{Examples}\label{Sec:ExliftRd}

We resume Examples~\ref{Ex:basic},~\ref{Ex:Sphere},~\ref{Ex:cayley} to construct the lifts of discretization maps described in the previous sections. In other words, we define discretization maps on $TQ$ and $T^*Q$ starting from a discretization map on $Q$.

\begin{example} \label{Ex:LiftBasic}
We focus now on the mid-point rule described in Example~\ref{Ex:basic} to define the tangent and cotangent lift of that symmetric discretization  map.
Assume that $Q$ is a vector space and let $R_d\colon TQ \rightarrow Q\times Q$ be the discretization map induced by the mid-point rule as follows $R_d(q, v)=\left(q-\frac{1}{2}v, q+\frac{1}{2}v\right)$. If we compute the inverse map $R_d^{-1}(q_0,q_1)=\left(\dfrac{q_0+q_1}{2},q_1-q_0 \right)$, we construct the sequence of points that will be used either for optimization or numerical integration as the discrete flow 
\begin{eqnarray*} \phi_d \colon Q\times Q &\rightarrow & Q\\
(q_0,q_1)& \mapsto & \left(\tau_Q\circ R_d^{-1} \right)(q_0,q_1)=\tau_Q \left(\dfrac{q_0+q_1}{2},q_1-q_0\right)=\dfrac{q_0+q_1}{2}\,. 
\end{eqnarray*}
Thus, the mid-point rule is recovered. 

To define the tangent lift of the discretization map  $R_d^{T}\colon TTQ \rightarrow TQ\times TQ$ on $TQ$ we first need to compute the tangent map whose matrix is 
$${\rm D} R_d= \begin{pmatrix} {\rm Id}  & -\dfrac{1}{2}\,  { \rm Id}\\ { \rm Id}  & \dfrac{1}{2}\,  {\rm Id} \end{pmatrix}\, .$$ 

The tangent lift of $R_d$ is given by:
\begin{align*}R_d^{T}(q,\dot{q},v,\dot{v})&=\left(TR_d \circ \kappa_Q\right) (q,\dot{q},v,\dot{v})=TR_d(q,v;\dot{q},\dot{v})\\&=\left(q-\dfrac{1}{2}\, v,  q+\dfrac{1}{2}\, v; \;\dot{q}-\dfrac{1}{2}\, \dot{v},  \dot{q}+\dfrac{1}{2}\, \dot{v},\right) \quad 
\\&\equiv\left(q-\dfrac{1}{2}\, v, \dot{q}-\dfrac{1}{2}\, \dot{v}; \; q+\dfrac{1}{2}\, v,  \dot{q}+\dfrac{1}{2}\, \dot{v},\right) \,,
\end{align*}
where we naturally identify elements of $T(Q\times Q)$ with elements of $TQ\times TQ$. 

We can also compute the inverse map:
\begin{equation*}
\left( R_d^{T}\right)^{-1}(q_0,v_0;q_1,v_1)=\left(\dfrac{q_0+q_1}{2},\dfrac{v_0+v_1}{2}; q_1-q_0,v_1-v_0 \right)\, .
\end{equation*} 

To compute the cotangent lift of $R_d$, that it, $R_d^{T^*}\colon TT^*Q\rightarrow T^*Q\times T^*Q$, we first need the tangent map of the inverse map $R^{-1}_d(q_0,q_1)=\left(\dfrac{q_0+q_1}{2},q_1-q_0\right)$:
$$D R^{-1}_d=\begin{pmatrix} \dfrac{1}{2}\, { \rm Id} & \dfrac{1}{2}\, { \rm Id} \\ - {\rm Id} & {\rm Id} \end{pmatrix}\, .$$
Thus the cotangent lift of $R_d^{-1}$ is given by:
\begin{align*}
\widehat{R_d}(q,v,p_q,p_v)&=\left(R_d(q,v), (p_q,\, p_v) (D_{(q,v)}R_d^{-1})\right)\\&= \left( q-\dfrac{1}{2}\, v, q+\dfrac{1}{2}\, v; \; \dfrac{p_q}{2}-p_v, \dfrac{p_q}{2}+p_v\right)\, .
\end{align*}
Finally, the cotangent lift of $R_d$ is the following discretization map on $T^*Q$:
\begin{align}\label{eq:RT*midpoint} R_d^{T^*}(q,p,\dot{q},\dot{p})&=\left(\Phi^{-1}\circ \widehat{R_d} \circ \alpha_Q \right)(q,p,\dot{q},\dot{p})=\left(\Phi^{-1}\circ \widehat{R_d}\right) (q,\dot{q},\dot{p},p)\nonumber \\ &= \Phi^{-1} \left( q-\dfrac{1}{2}\,\dot{q}, q+\dfrac{1}{2}\, \dot{q}; \; \dfrac{\dot{p}}{2}-p, \dfrac{\dot{p}}{2}+p\right) \nonumber\\&=\left( q-\dfrac{1}{2}\,\dot{q}, p-\dfrac{\dot{p}}{2}; \; q+\dfrac{1}{2}\, \dot{q}, p+\dfrac{\dot{p}}{2}\right) \, . \end{align}

The inverse map $\left(R_d^{T^*}\right)^{-1}\colon T^*Q\times T^*Q \rightarrow TT^*Q$ is given by:
\begin{equation}\label{eq:RT*midpointInv}\left(R_d^{TT^*Q}\right)^{-1}(q_0,p_0; q_1,p_1)=\left( \dfrac{q_0+q_1}{2},\frac{p_0+p_1}{2},q_1-q_0, p_1-p_0\right)\, . \demo \end{equation}
\end{example}

 \begin{example}\label{Ex:liftSphere} Let us lift the discretization map in Example~\ref{Ex:Sphere}. To simplify the computations we consider the discretization map that fixes the first point (compared with~\eqref{eq:RdS2mid}) as follows
 	\[
 	R_d (x, \xi)=\left(x, \cos \left( \|\xi\|\right) x+\sin \left(\|\xi\|\right) \frac{\xi}{\|\xi\|}\right)\, .
 	\]
 	Remember that $x\,\cdot  x^T=1$ and $x\, \cdot \xi^T=0$ because the manifold $Q$ is the sphere $S^2$.
 	
 The tangent map of $R_d$ is given by the matrix
 \begin{equation*}
 D_{(x,\xi)} R_d=\begin{pmatrix}  {\rm Id} & \rvline & 0 \\ \hline  \cos\left(\|\xi\|\right) \, { \rm Id}& \rvline & N(x,\xi)\end{pmatrix}\, ,
 \end{equation*}
 where 
  \begin{equation*}
 N_{ij}=-\sin \left(\|\xi\|\right) \, \dfrac{\xi_j x_i}{\|\xi\|}+ \cos \left(\|\xi\| \right) \, \dfrac{\xi_j \xi_i}{\|\xi\|}+\sin \left(\|\xi\|\right) \, \cdot  \begin{cases}  \dfrac{\|\xi\|^2-\xi_i\xi_i}{\|\xi\|^3}, \quad \text{for} \; i=j\, ,\\ \\
 \dfrac{-\xi_i\xi_j}{\|\xi\|^3}, \quad \text{for} \; i\neq j\, ,\\
 \end{cases} 
 \end{equation*} are the entries of the invertible matrix $N(x,\xi)$.
 Thus, the tangent lift of $R_d$ is the following discretization map on $TQ$:
 \begin{align*}
 R_d^{T}(x,\dot{x},\xi,\dot{\xi})&= \left(TR_d \circ \kappa_Q\right) (x,\dot{x},\xi,\dot{\xi})=TR_d(x,\xi; \dot{x},\dot{\xi})\\
 &=\left(x, \cos \left( \|\xi\|\right) x+\sin \left(\|\xi\|\right) \frac{\xi}{\|\xi\|}, \dot{x}, \cos(\|\xi\|) \, \dot{x}+ N(x,\xi) \, \dot{\xi}\right)  \\
 &\equiv  \left(x, \dot{x} ; \; \cos \left( \|\xi\|\right) x+\sin \left(\|\xi\|\right) \frac{\xi}{\|\xi\|},  \cos(\|\xi\|) \, \dot{x}+ N(x,\xi) \, \dot{\xi}\right) \, .
 \end{align*}
 To compute the cotangent lift of $R_d$, $R_d^{T^*}\colon TT^*Q\rightarrow T^*Q\times T^*Q$, we first need  the tangent map of the inverse map $R^{-1}_d$ or, equivalently, the inverse of the tangent map:
 $$D R^{-1}_d=\left(D R_d\right)^{-1}=\begin{pmatrix} {\rm Id} & 0 \\ -\cos(\|\xi\|) \, N^{-1} & N^{-1}  \end{pmatrix}\,  .$$
 Thus the cotangent lift of $R_d^{-1}$ is given by:
 \begin{align*}
 \widehat{R_d}(x,\xi ,p_x,p_{\xi})&=\left(R_d(x,\xi), (p_x,p_{\xi}) \left(D_{(x,\xi)} R_d\right)^{-1}\right)\, .
 \end{align*}
 Finally, the discretization map on $T^*Q$ is obtained as follows:
 \begin{align*} R_d^{T^*}(x,p,\dot{x},\dot{p})&=\left(\Phi^{-1}\circ \widehat{R_d} \circ \alpha_Q \right)(x,p,\dot{x},\dot{p})=\left(\Phi^{-1}\circ \widehat{R_d}\right) (x,\dot{x},\dot{p},p)\\ &= \Phi^{-1} \left(x, \cos \left( \|\dot{x}\|\right) x+\sin \left(\|\dot{x}\|\right) \frac{\dot{x}}{\|\dot{x}\|}\, ; \, \dot{p}-\cos \left( \|\dot{x}\|\right) \, p N^{-1}, p\, N^{-1}\right)\, \\
 &\equiv \left(x,-\dot{p}+\cos \left( \|\dot{x}\|\right) \, p N^{-1}\, ; \,  \cos \left( \|\dot{x}\|\right) x+\sin \left(\|\dot{x}\|\right) \frac{\dot{x}}{\|\dot{x}\|}\, , p\, N^{-1} \right)\,.\demo
  \end{align*}

 \end{example}

 \begin{example}\label{Ex:liftCayley} 
	Let us lift the discretization map in Example~\ref{Ex:cayley}. As in the previous example,  to simplify the computations we consider the discretization map $R_{d,{\rm cay}}\colon TSO(3)\rightarrow SO(3)\times SO(3)$ that fixes the first point (compared with~\eqref{eq:Rcaymid}) as follows:
	\begin{equation*}
	R_{d,{\rm cay}}(A,X)= \left(A, A\, {\rm cay} (A^T\,X)\right)=\left(A,A \left({\rm Id}_3-A^TX/2\right)^{-1}\left({\rm Id}_3+A^TX/2\right)\right)\,.
	\end{equation*}
	
	The tangent map of $R_{d,{\rm cay}}$ is given by the matrix
	\begin{align*}
	D_{(A,X)} 	R_{d,{\rm cay}}&=\begin{pmatrix} {\rm Id} & \rvline & 0 \\ \hline  {\rm cay}(A^TX)+A \, \dfrac{\rm d}{{\rm d}A} {\rm cay} (A^TX) \, X  & \rvline & A \, \dfrac{\rm d}{{\rm d }X} {\rm cay} (A^TX) \, A^T\end{pmatrix}\\ &= \begin{pmatrix}  {\rm Id} & \rvline & 0 \\ \hline  M  & \rvline & N \end{pmatrix}\, . 
	\end{align*}
	
	Thus, 
	\begin{align*}
	R_{d,{\rm cay}}^{T}(A,\dot{A},X,\dot{X})&= \left(TR_{d,{\rm cay}} \circ \kappa_Q\right) (A,\dot{A},X,\dot{X})=TR_d(A,X,\dot{A},\dot{X})\\
	&=\left(A, A\, {\rm cay} (A^TX), \dot{A}, M \dot{A}+ N \,\dot{X}\right)\\
	&\equiv \left(A, \dot{A}; A \,{\rm cay} (A^TX),  M \dot{A}+ N \,\dot{X}\right)\, ,
	\end{align*}
	where $N$ is an invertible matrix.
	To compute the discretization map $R_{d,{\rm cay}}^{T^*}\colon TT^*Q\rightarrow T^*Q\times T^*Q$ on $T^*Q$ as the cotangent lift of $R_d$, we first need  the tangent map of the inverse map $R^{-1}_{d,{\rm cay}}$ or, equivalently, the inverse of the tangent map:
	$$D R^{-1}_{d,{\rm cay}}=\left(D R_{d,{\rm cay}}\right)^{-1}=\begin{pmatrix} {\rm Id} & 0 \\ -N^{-1}M & N^{-1}  \end{pmatrix}\,  .$$

	Thus the cotangent lift of $R^{-1}_{d,{\rm cay}}$ is given by:
	\begin{align*}
	\widehat{R_{d,{\rm cay}}}(A,X,p_A,p_X) &=\left(R_{d,{\rm cay}}(A,X), (p_A,p_X) \left(D R_{d,{\rm cay}}\right)^{-1}(A,X)\right)\, .
	\end{align*}
	Finally, the discretization map on $T^*Q$ is obtained as follows:
	\begin{align*} R_{d,{\rm cay}}^{T^*}(A,p,\dot{A},\dot{p})&=\left(\Phi^{-1}\circ \widehat{R_{d,{\rm cay}}} \circ \alpha_Q \right)(A,p,\dot{A},\dot{p})=\left(\Phi^{-1}\circ \widehat{R_{d,{\rm cay}}}\right) (A,\dot{A},\dot{p},p)\\ &= \Phi^{-1} \left(A, A\,  {\rm cay} (A^T\dot{A}), \dot{p}-pN^{-1}M,pN^{-1}\right)\\ &=  \left(A, pN^{-1}M-\dot{p};A \, {\rm cay} (A^T\dot{A}),pN^{-1}\right) \,. \demo
	\end{align*}
\end{example}

\section{discretization maps associated to SODEs}\label{Sec:AllSODE}

The tangent lift of discretization maps defined in Section~\ref{Sec:liftRd} appears naturally when geometrically designing discretizations of second order differential equations (SODEs). Remember that a second order differential equation is  a vector field $\Gamma$ such that  
$\tau_{TQ}(\Gamma)=T\tau_{Q}(\Gamma)$.
This implies that the vector field $\Gamma$ on $TQ$ is a section of the second order tangent bundle $T^{(2)}Q$, as described in~\cite{BookLeon}. 
Locally, if we take coordinates $(q^i)$ on $Q$ and induced coordinates $(q^i, \dot{q}^i)$ on $TQ$, then 
\[
\Gamma=\dot{q}^i\frac{\partial}{\partial q^i}+\Gamma^i(q, \dot{q})\frac{\partial}{\partial \dot{q}^i}\; .
\]
To find the integral curves of $\Gamma$ is equivalent to solve the following system of second order differential equations: 
\begin{equation*}
\frac{d^2 q^i}{dt^2}=\Gamma^i\left(q, \frac{dq}{dt}\right)\; .
\end{equation*}

Now, we want to discretize these equations using the notion of discretization map defined on $TQ$ as in Definition~\ref{def:DiscreteMap2}. Here we have two options: we could directly define a discretization map on $TQ$ denoted by $R_d^{TTQ}\colon TTQ \rightarrow TQ\times TQ$ or we could tangently lift a discretization map on $Q$ to obtain $R_d^T\colon TTQ \rightarrow TQ\times TQ$  as defined in Proposition~\ref{prop-liftTR}. 

Let us consider in general that we have a discretization map on $TQ$,
\[
R_d^{TTQ}: TTQ\rightarrow TQ\times TQ\, ,
\]
given by 
$
R_d^{TTQ}(q, v, \dot{q}, \dot{v})= \left(\left(R^{TTQ}\right)^1 (q, v, \dot{q}, \dot{v}), \left(R^{TTQ}\right)^2 (q, v, \dot{q}, \dot{v})\right)
$. Note that
 $\left(R^{TTQ}\right)^i (q, v, \dot{q}, \dot{v})\in TQ$ for $i=1,2$.

A first option for discretizing a SODE $\Gamma$ consists of the following implicit discrete equation:
\begin{equation}\label{eq:first-method}
\left(\left(R^{TTQ}\right)^2\circ h\Gamma\right)(q_k, v_k)=\left(\left(R^{TTQ}\right)^1\circ h\Gamma\right)(q_{k+1}, v_{k+1})\, ,
\end{equation}
where $h$ is a positive small real number that determines the step size.
The numerical method starts from the initial data $v_k\in T_{q_k} Q$, then the Equation~\eqref{eq:first-method} is solved implicitly to obtain $v_{k+1}\in T_{q_{k+1}} Q$. 
Section~\ref{Sec:Newmark} shows that a discretization map on $TQ$, not coming from a tangent lift, recovers Newmark method using the discretization method in Equation~\eqref{eq:first-method}.
Geometrically, these methods given in Equation~\eqref{eq:first-method} are based on the structure of  groupoid of an implicit difference equation, in this case $TQ\times TQ\rightrightarrows TQ$ (see \cite{IgMaMaPa} for more details). 

A second option for discretizing a SODE consists of the following numerical scheme:
\begin{equation}\label{eq:2ndmethod-SODE}
h\, \Gamma \left(\left(\tau_{TQ}\circ \left(R^{TTQ}_d\right)^{-1}\right)(q_k, v_k; q_{k+1}, v_{k+1})\right)=\left(R^{TTQ}_d\right)^{-1}(q_k, v_k; q_{k+1}, v_{k+1})\,. 
\end{equation}
As in Equation~\eqref{eq:first-method}, the numerical method is usually implicit. We will focus on this discretization process in Sections~\ref{Sec:SODE} and ~\ref{Sec:GeomInt} when constructing geometric integrators for mechanical systems.

Let us do a simple example to show that the numerical schemes in Equations~\eqref{eq:first-method} and~\eqref{eq:2ndmethod-SODE} are usually different.

\begin{example}
Consider the discretization map on $TQ$ obtained from a tangent lift in Example \ref{Ex:LiftBasic} and the inverse map: 
\begin{align*}&
R_d^T(q,\dot{q},v,\dot{v}) =\left(q-\dfrac{1}{2}\, v, \dot{q}-\dfrac{1}{2}\, \dot{v}\,  ; 
q+\dfrac{1}{2}\, v\, , \dot{q}+\dfrac{1}{2}\, \dot{v}\right) \,,\\ &
\left(R_d^T\right)^{-1}(q_0,v_0;q_1,v_1)=\left(\dfrac{q_0+q_1}{2},\dfrac{v_0+v_1}{2}; q_1-q_0,v_1-v_0 \right)\, .
\end{align*}
The method in Equation~\eqref{eq:first-method} becomes: 
\begin{align*}
\frac{q_{k+1}-q_{k}}{h}&= \frac{v_k+v_{k+1}}{2}\, ,\\
\frac{v_{k+1}-v_{k}}{h}&= \frac{1}{2}\left(\Gamma (q_k, v_k) +\Gamma (q_{k+1}, v_{k+1})\right)\, .
\end{align*}
However, for the same discretization map on $TQ$ the method in Equation~\eqref{eq:2ndmethod-SODE} is given by the following  equations: 
\begin{align*}
\frac{q_{k+1}-q_{k}}{h}&= \frac{v_k+v_{k+1}}{2}\, ,\\
\frac{v_{k+1}-v_{k}}{h}&= \Gamma\left( \frac{q_k+q_{k+1}}{2}, 
 \frac{v_k+v_{k+1}}{2}\right)\, . \demo
\end{align*}
\end{example}

	\subsection{Newmark method from a discretization map} \label{Sec:Newmark}
An example of discretization using Equation~\eqref{eq:first-method} is the Newmark method~\cite{Newmark}, a classical time-stepping method very  common in structural mechanical codes. For simplicity, we consider a typical mechanical Lagrangian $L: T{\mathbb R}^n\longrightarrow {\mathbb R}$:
\[
L(q, \dot{q})=\frac{1}{2}\dot{q} M\dot{q}^T-V(q)\, ,
\]
where $(q, \dot{q})\in T{\mathbb R}^n$, $M$ is a positive definite constant matrix and $V$ is a potential function.
The corresponding 
Euler-Lagrange equations are: 
\begin{equation}\label{eq:NewmarkL}
\ddot{q}=-M^{-1}\nabla V(q)\,,
\end{equation}
where $\nabla$ denotes the gradient of the potential function.

The Newmark methods are widely used in simulations of such mechanical systems, including even external forces \cite{KaMaOrWe}. To construct the method two real parameters $\alpha$ and $\beta$ are selected so that the algorithm determines $(q_{k+1}, \dot{q}_{k+1})$ in terms of $(q_{k}, \dot{q}_{k})$ as follows: 
\begin{align}\label{eq:NewmarkMethod}
	q_{k+1}&=q_k+h\dot{q}_k+\frac{h^2}{2}\left( (1-2\beta) a_{k}+2\beta a_{k+1}\right)\\
	\dot{q}_{k+1}&=\dot{q}_k+h\left( (1-\gamma) a_k+\gamma a_{k+1}\right)\, , \nonumber
\end{align}
where $a_k=-M^{-1}\nabla V(q_k)$ and  $a_{k+1}=-M^{-1}\nabla V(q_{k+1})$.

We show here that the family of Newmark methods can be obtained from a discretization map on the tangent bundle $TQ$. Let us define $\left(R^{TTQ}_d\right): TT{\mathbb R}^n\equiv {\mathbb R}^{4n}\rightarrow T{\mathbb R}^n\times T {\mathbb R}^n\equiv {\mathbb R}^{2n}\times {\mathbb R}^{2n}$ by
\begin{align*}
\left(R^{TTQ}\right)^1(q, v, \dot{q}, \dot{v})&=\left(q-\frac{1}{2}\dot{q}+\frac{h}{2}(\gamma-2\beta)\dot{v}, v-\gamma \dot{v}\right)\, , \\
\left(R^{TTQ}\right)^2(q, v, \dot{q}, \dot{v}) &=\left(q+\frac{1}{2}\dot{q}+\frac{h}{2}(\gamma-2\beta)\dot{v}, v+(1-\gamma) \dot{v}\right)\, .
\end{align*}
The Jacobian matrix of $R_d^{TTQ}$ is
\[
\left(
\begin{array}{rrrr}
{\rm Id} &0&-\frac{1}{2}\,{\rm Id} &\frac{h}{2}(\gamma-2\beta) \,{\rm Id}\\
0& {\rm Id}&0&-\gamma \,{\rm Id}\\
{\rm Id} &0&\frac{1}{2} \,{\rm Id}&\frac{h}{2}(\gamma-2\beta) \,{\rm Id}\\
0&{\rm Id}&0&(1-\gamma) \,{\rm Id}
\end{array}
\right)\, .
\]
It is straightforward that $R_d^{TTQ}$ satisfies both properties in Definition~\ref{def:DiscreteMap2}. Hence, $R_d^{TTQ}$ is an discretization map on $TQ$. 

The Euler-Lagrange equations~\eqref{eq:NewmarkL} can be rewritten as the submanifold $S$ of $T^{(2)}Q\subset TTQ$,
\[
S=\{(q, \dot{q}, a ) \; \mid\; a=-M^{-1}\nabla V(q)   \}\, ,
\]
with the natural inclusion $i: T^{(2)}Q\hookrightarrow TTQ$, $i(q, \dot{q}, a )=(q, \dot{q}, \dot{q}, a )$.

Hence,  the dynamics induced by the Newmark method is equivalent to the following algorithm:
\begin{enumerate}
	\item Take an initial  position and velocity $(q_k, \dot{q}_k)$.
	\item Evaluate $a_k=-M^{-1}\nabla V(q_k)$.
	\item Solve the system obtained from Equation~\eqref{eq:first-method}:
	\begin{equation}\label{aqr}
\left(R^{TTQ}\right)^2(q_k, \dot{q}_k; h\dot{q}_k, ha_k)=		\left(R^{TTQ}\right)^1(q_{k+1}, \dot{q}_{k+1}; h\dot{q}_{k+1}, h a_{k+1})\, ,
	\end{equation}
	where $	a_{k+1}=-M^{-1}\nabla V(q_{k+1})$.
\end{enumerate}
Observe that Equation~\eqref{aqr} is equal to 
\begin{eqnarray*}
	q_k+\frac{h}{2}\dot{q}_k+\frac{h^2}{2}(\gamma-2\beta)a_k&=&
	q_{k+1}-\frac{h}{2}\dot{q}_{k+1}-\frac{h^2}{2}(\gamma-2\beta)a_{k+1}\, ,\\
	\dot{q}_k+h(1-\gamma)a_k&=&\dot{q}_{k+1}-h\gamma a_{k+1}\, .
\end{eqnarray*}
After algebraic manipulations, the above equations are equivalent to the well-known Newmark method in Equation~\eqref{eq:NewmarkMethod}.

Note that if $\gamma=1/2$ and $\beta=1/4$, then $R_d^{TTQ}$ is precisely the tangent lift of the discretization map on $Q$ coming from the mid-point rule as described in Example~\ref{Ex:LiftBasic}.

\subsection{Discretization maps associated with discrete second order equations} 
\label{Sec:SODE}

In this section we briefly discuss the possibility to find a discrete version of  a second order differential equation (SODE) using a  second order discrete equation (SOdE).

According to \cite{MW_Acta}, a SOdE is given by a map
$\Gamma_d: Q\times Q\rightarrow Q\times Q\times Q\times Q$ such that 
\[
\Gamma_d(q_{k-1}, q_k)=(q_{k-1}, q_k, q_k, \tilde{\Gamma}_d(q_{k-1}, q_k))\,,
\]
in other words, $q_{k+1}=\tilde{\Gamma}_d(q_{k-1}, q_k)$. From two initial conditions $q_0$, $q_1$ this equation defines the discrete evolution as the sequence $\{q_0, q_1, q_2, \ldots\}$.

Given a discretization map on $Q$ and a second order vector field $\Gamma$, we wonder if, under any assumption, the tangent lift of the discretization map, $R_d^T$, could define a discrete second order equation $\Gamma_d$.  The specific question is: When does a discretization map $R_d$ make Diagram~\eqref{DiagramRd2} commutative? 

\begin{equation}\label{DiagramRd2}
\xymatrix{& TTQ \ar[d]^{\tau_{TQ}} \ar[rr]^{R_d^{T}} && TQ\times TQ   \ar[d]^{(R_d,R_d)} \\TQ \ar[r]^{{\rm Id}_{TQ}} \ar[ur]^\Gamma & TQ  \ar[d]<2pt>^{R_d} && Q\times Q \times Q \times Q \\    & Q\times Q \ar[u]<2pt>^{R_d^{-1}} \ar[urr]^{\Gamma_d }&&  }
\end{equation}

\begin{proposition} Let $Q$ be a vector space and $\Gamma$ be a SODE.  If $R_d\colon TQ \rightarrow Q\times Q$ is the discretization map on $Q$ defined from the $\theta$-method: 
	\begin{equation}\label{eq:thetaSode}R_d(q,v)=(q-\theta \, v, q+ (1-\theta)\, v),\end{equation}
	then Diagram~\eqref{DiagramRd2} is commutative, that is, $$(R_d,R_d)\circ R_d^{T}\circ \Gamma \circ R_ d^{-1}\colon Q\times Q \rightarrow Q\times Q \times Q \times Q$$ defines a second order discrete equation (SOdE).
	\end{proposition}
\begin{proof}
Let $(q,v)$ local coordinates for $TQ$. 
Let $R_d(q,v)=\left(R^1_d(q,v),R^2_d(q,v)\right)$, we compute
\begin{equation*}
 \left(R_d^{T}\circ \Gamma\right)(q,v)=\left(T_{(q,v)}R^1_d (\Gamma(q,v)), T_{(q,v)}R^2_d (\Gamma(q,v))\right).
\end{equation*}
If we apply now $(R_d,R_d)\colon TQ\times TQ \rightarrow Q\times Q \times Q\times Q$, the resulting expression defines a SOdE if and only if the second and third component are equal, that is, 
\begin{equation}\label{eq:RdSODE} R^2_d(T_{(q,v)}R^1_d (\Gamma(q,v)))=R^1_d(T_{(q,v)}R^2_d (\Gamma(q,v))).\end{equation}
It is a straightforward computation to verify that the discretization maps defined from the $\theta$-method in~\eqref{eq:thetaSode} satisfy Equation~\eqref{eq:RdSODE}. 
\end{proof}

In fact, the above proposition could be stated more generally. Any discretization map that satisfies Equation~\eqref{eq:RdSODE} defines a SOdE by the tangent lift of that map. 

Equation~\eqref{eq:RdSODE} is equivalent to the commutativity of the following diagram:
\begin{equation}\label{DiagramConditiondSODE}
\xymatrix{ T^{(2)}Q\subset TTQ  \ar[rr]^{TR^2_d} && TQ \ar[d]_{R^1_d} \\ TQ \ar[u]^{\Gamma} \ar[d]_{\Gamma} && Q \\ T^{(2)}Q\subset TTQ  \ar[rr]^{TR^1_d} && TQ \ar[u]^{R^2_d} 
}
\end{equation}

In particular, if the discretization map $R_d: TQ\rightarrow Q\times Q$ is defined from a standard retraction map as in Definition \ref{def-RetractMap}, that is, $R_d^1=\tau_Q$, then  Diagram~\eqref{DiagramConditiondSODE} is always commutative
\begin{equation*}
\xymatrix{ T^{(2)}Q\subset TTQ  \ar[rr]^{TR^2_d} && TQ \ar[d]_{\tau_Q} \\ TQ \ar[u]^{\Gamma} \ar[d]_{\Gamma} \ar[rr]^{\tau_Q} && Q \\ T^{(2)}Q\subset TTQ  \ar[rr]^{TR^1_d} && TQ \ar[u]^{R^2_d} 
}
\end{equation*}
since $(\tau_Q\circ TR_d^2)(\Gamma)=
(R_d^2\circ T\tau_Q)(\Gamma)$.  Therefore, any standard retraction map defines a SOdE $\Gamma_d$.

\section{Construction of geometric integrators from discretization maps} \label{Sec:GeomInt}

In this section we describe how geometric integrators are obtained for both Hamiltonian and Euler-Lagrange equations by discretizing their equations using discretization maps. In Section~\ref{Sec:discreteVarCalc}, we establish the relation with discrete variational calculus  where  the variational principles are discretized to obtain the discrete flow (see \cite{MW_Acta}). 

In Section~\ref{Sec:GeomIntHam} we look at the Hamiltonian framework \cite{AbMa}. Hamiltonian systems have the property that the associated flow is a symplectic transformation. As described in~\cite{sanz-serna,hairer,blanes}, it is important to define numerical methods that also preserve that property. Remember that a numerical one-step method is called symplectic if the one-step map, in other words, the discrete flow, is symplectic whenever the method is applied to a smooth Hamiltonian system. 

Second, we describe geometric integrators obtained from the Lagrangian viewpoint in Section~\ref{Sec:GeomIntLag}. 

In order to describe Hamiltonian and Lagrangian mechanics, we consider the symplectic manifold $(T^*Q,\omega_Q)$ that has the musical isomorphisms
$\omega_Q^\flat: {\mathfrak X}(T^*Q)\rightarrow \Omega^1 (T^*Q)$ defined by $\omega_Q^\flat(X)=\alpha$ where 
$i_X\omega_Q=\alpha$ (see for instance  \cite{LiMarle}). The inverse of $\omega_Q^\flat$ is denoted by $\omega_Q^\sharp$, that is, $\omega_Q^\sharp=(\omega_Q^\flat)^{-1}$.

\subsection{Geometric integrators in Hamiltonian framework}\label{Sec:GeomIntHam}

 Let  $H: T^*Q\rightarrow {\mathbb R}$ be a Hamiltonian function with corresponding Hamiltonian vector field $X_H$ derived from Hamilton's equations: 
 \[
 i_{X_H}\omega_Q=dH\, .
 \]
 The triple $(T^*Q,\omega_Q,H)$ defines a Hamiltonian system. 
 Equivalently, an integral curve of $X_H$ is solution to Hamilton's equations:
 \[
 \frac{dq^i}{dt}=\frac{\partial H}{\partial p_i}\; ,\qquad \frac{dp_i}{dt}=-\frac{\partial H}{\partial q^i}\, ,
 \]
 where $(q^i, p_i)$ are canonical coordinates on $T^*Q$ (see \cite{AbMa}). 
 In other words,  a  solution $\gamma\colon I \rightarrow T^*Q$  of Hamilton's equations must satisfy
\[\omega_Q^\flat \left(\dot{\gamma}(t)\right) ={\rm d} H(\gamma(t))\,, \mbox{  equivalently } \dot{\gamma}(t) =\omega^\sharp \left({\rm d} H(\gamma(t))\right) .\]

A discretization map on $T^*Q$, that is, $R^{TT^*Q}_d\colon TT^*Q\rightarrow T^*Q\times T^*Q$ defines the following numerical integrator for step size $h$:
\begin{equation} \label{Eq:HMethod}
\left(R^{TT^*Q}_d\right)^{-1}(q_0, p_0; q_1, p_1)=\omega^\sharp \left( h\, {\rm d}H\left(\left(\tau_{T^*Q}\circ \left(R^{TT^*Q}_d\right)^{-1}\right)(q_0, p_0; q_1, p_1)\right)\right) \, .
\end{equation}

Equivalently, similar to Equation~\eqref{eq:2ndmethod-SODE}, we have
\begin{equation}
\label{Eq:HMethod2}
h\, X_H \left(\left(\tau_{T^*Q}\circ \left(R^{TT^*Q}_d\right)^{-1}\right)(q_0, p_0; q_1, p_1)\right) =   \left(R^{TT^*Q}_d\right)^{-1}(q_0, p_0; q_1, p_1) .
\end{equation}

This numerical integrator may be defined for any discretization map on $T^*Q$. However, if such a map is the cotangent lift of a discretization map on $Q$ (see Section~\ref{Sec:coTliftRd}), then the numerical integrator is symplectic as stated in the following proposition. 

For the proof we need to recall the notion of a Lagrangian submanifold of a symplectic manifold $(M, \omega)$.

An immersed submanifold $N$ of $M$, or immersion,
$f : N  \rightarrow M$ is  Lagrangian if so is the space $Tf (T_x N)$  as a subspace of $T_{f(x)} M$ for each point $x \in N$, that is, $Tf (T_x N)=(Tf (T_x N))^\perp$ where $^\perp$ denotes the orthogonal complement of the subspace with respect to the symplectic form.  Note that an immersion $f\colon N\rightarrow M$ is Lagrangian if and only if $f^*\omega=0$ and the dimension of $N$ is half the dimension of $M$. The most common way to define a Lagrangian submanifold of a symplectic manifold is as graph of a closed one-form.

\begin{proposition}\label{Prop:HsymplecticMethod} Let $R_d\colon TQ\rightarrow Q\times Q$ be a discretization map on $Q$ and $H: T^*Q\rightarrow {\mathbb R}$ be a Hamiltonian function. Equation~\eqref{Eq:HMethod} written for the cotangent lift of $R_d$, that is, $R_d^{T^*}$, defines a symplectic integrator of the Hamiltonian system $(T^*Q,\omega_Q, H)$. 
\end{proposition}

\begin{proof}
 As the submanifold ${\rm d}H\left(\left(\tau_{T^*Q}\circ \left(R_d^{T^*}\right)^{-1}\right)(q_0, p_0; q_1, p_1)\right)$ in Equation~\eqref{Eq:HMethod} is Lagrangian on $(T^*T^*Q, \omega_{T^*Q})$,	Proposition~\ref{prop:sympl} guarantees that Equation~\eqref{Eq:HMethod} determines a Lagrangian submanifold of $(T^*Q\times T^*Q, \Omega_{12})$. Locally, such a manifold can be expressed as the graph of a local symplectomorphism $\varphi_h\colon T^*Q\rightarrow T^*Q$, see~\cite{LiMarle} for more details. Consequently, the geometric method obtained from Equation~\eqref{Eq:HMethod} is symplectic. 
\end{proof}

Let us use the above result to obtain some of the symplectic numerical methods known in the literature.

\begin{example}
 Let $H\colon T^*Q\rightarrow \mathbb{R}$ be a Hamiltonian function, the cotangent lift~\eqref{eq:RT*midpoint},~\eqref{eq:RT*midpointInv} of the discretization map associated to the mid-point rule in Example~\ref{Ex:LiftBasic} used in Equation~\eqref{Eq:HMethod} leads to the following equations:
\begin{align*}\left(\frac{q_0+q_1}{2}, \right.&  \left. \frac{p_0+p_1}{2}, q_1-q_0, p_1-p_0\right)=\left(\frac{q_0+q_1}{2}, \frac{p_0+p_1}{2}, \right.\\ & \left. h\frac{\partial H}{\partial p}\left(\frac{q_0+q_1}{2}, \frac{p_0+p_1}{2}\right), 
-h\frac{\partial H}{\partial q}\left(\frac{q_0+q_1}{2}, \frac{p_0+p_1}{2}\right)\right)
\; .\end{align*}
Equivalently, the equations describe the following symplectic integrator:
\begin{eqnarray*}
\frac{q_1-q_0}{h}&=&\frac{\partial H}{\partial p}\left(\frac{q_0+q_1}{2}, \frac{p_0+p_1}{2}\right)\, ,\\
\frac{p_1-p_0}{h}&=&-\frac{\partial H}{\partial q}\left(\frac{q_0+q_1}{2}, \frac{p_0+p_1}{2}\right)\, .
\end{eqnarray*}
The above integrator corresponds with an implicit second-order symplectic method with initial condition $(q_0,p_0)$.
\demo
\end{example}

\begin{example}
The discretization map on $Q$,
		$
		R_d(q,v)=(q-v,q)\,,
		$ is lifted to the cotangent bundle as follows
\begin{equation*}
\begin{array}{cccc} R_d^{T^*}\colon & TT^*Q & \longrightarrow & T^*Q \times T^*Q \\ & (q,p,\dot{q},\dot{p}) &\longmapsto & (q-\dot{q}, p, q,p+\dot{p}) \, .\end{array}
\end{equation*}	
As $\left(R_d^{T^*}\right)^{-1}(q_0,p_0,q_1,p_1)=(q_1,p_0,q_1-q_0,p_1-p_0)$, Equation~\eqref{Eq:HMethod} leads to the following symplectic method:
\begin{eqnarray*}
	\frac{q_1-q_0}{h}&=&\frac{\partial H}{\partial p}\left(q_1,p_0\right) \, ,\\
	\frac{p_1-p_0}{h}&=&-\frac{\partial H}{\partial q}\left(q_1,p_0\right) \, .
\end{eqnarray*}
For a Hamiltonian function $H(p,q)=\dfrac{1}{2}p M p^T+V(q)$, with a constant positive definite matrix $M$,  the integrator is an explicit symplectic method.  \demo
\end{example}

\begin{example}
Now consider a Hamiltonian function $H: T^*S^2\rightarrow {\mathbb R}$ on $T^*S^2$ that we identify with the tangent bundle $T S^2$: 
\[
T^*S^2\equiv \{(x, p)\in {\mathbb R}^3\times {\mathbb R}^3\; \mid\; \norm{x}=1,  \; x\cdot p=0\}\, .
\]
For the discretization of the corresponding Hamiltonian equations we will use the discretization map 
$R_d: TS^2\rightarrow S^2\times S^2$ given by
\[
R_d(x, \xi)=\left(x, \frac{x+\xi}{\norm{x+\xi}}\right), 
\]
whose inverse is precisely: 
\[
R_d^{-1}(x_0, x_1)=\left(x_0, \frac{x_1}{x_0\cdot x_1} -x_0 \right)\, ,
\]
whenever it is well defined. Now, we will compute the inverse of the cotangent lift of the discretization map given in Equation~\eqref{eq:localinverseRd*}, that is,	\begin{equation*}
	\left(R_d^{T^*}\right)^{-1}(x_0,p_0; x_1,p_1)= \alpha_Q^{-1}\left(R_d^{-1} (x_0,x_1), (-p_0, p_1)\, { D}_{R_d^{-1} (x_0,x_1)}R_d\right)\, . 
	\end{equation*}
Having in mind the definition of $T^*S^2$, it can be computed that the matrix ${D}_{R_d^{-1} (x,y)}R_d$ is equal to: 
\[
{D}_{R_d^{-1} (x,y)}R_d
=
\left(
\begin{array}{rr}
{\rm Id}_{3\times 3}&0\\
(x\cdot y)\, {\rm Id}_{3\times 3}& C
\end{array}
\right)
\]
where $C$ is the matrix with entries
\begin{equation*}
c_{ij}=\left\{ \begin{array}{lcl} (x\cdot y)\left[1+(x\cdot y)y_ix_i-y_i^2\right]
&\mbox{if} & i=j, \\
(x\cdot y)\left[(x\cdot y)y_ix_j-y_iy_j\right]&\mbox{if} &  i\not=j\, . \end{array}\right.
\end{equation*}
Therefore,
\[
	\left(R_d^{T^*}\right)^{-1}(x_0,p_0; x_1,p_1)=
	\left(x_0, p_1 C\, ; \,  \frac{1}{x_0\cdot x_1} x_1-x_0, -p_0+(x_0\cdot x_1)p_1\right)\, .
\]
As a result, we obtain the following symplectic integrator for Hamilton's  equations: 
\begin{align*}
  \frac{1}{x_k\cdot x_{k+1}} x_{k+1}-x_k&=h\frac{\partial H}{\partial p}(x_k, p_{k+1}C)\, ,\\
   -p_k+(x_k\cdot x_{k+1})p_{k+1}&=-h\frac{\partial H}{\partial q}(x_k, p_{k+1}C) \, .\demo
\end{align*}

\end{example}

\begin{remark}
Another option to construct geometric integrators is to use an expression similar to Equation~\eqref{eq:first-method} but now adapted for Hamiltonian vector fields, that is, 
\begin{equation*}
\left(\left(R_d^{TT^*Q}\right)^2\circ h X_H\right)(q_k, p_k)=\left(\left(R_d^{TT^*Q}\right)^1\circ h X_H\right)(q_{k+1}, p_{k+1})\; . \end{equation*}
Note that here the discretization map on $T^*Q$ does not have to be the cotangent lift of one on $Q$. However, even if the cotangent lift is considered, the method is not necessarily symplectic. For instance,  for the discretization map coming from the mid-point rule in Examples~\ref{Ex:basic} and~\ref{Ex:LiftBasic}, we obtain the symmetric  second-order method: 
\begin{align*}
\frac{q_{k+1}-q_k}{h}&= \frac{1}{2}\left(
\frac{\partial H}{\partial p}(q_k, p_k)+\frac{\partial H}{\partial p}(q_{k+1}, p_{k+1})
\right) \; ,\\
\frac{p_{k+1}-p_k}{h}&= -\frac{1}{2}\left(
\frac{\partial H}{\partial q}(q_k, p_k)+\frac{\partial H}{\partial q}(q_{k+1}, p_{k+1})
\right) \; .\\
\end{align*}
However, this method is not  symplectic because, in general, $dq_{k+1}\wedge d p_{k+1}-dq_{k}\wedge d p_{k}\not= 0$ when restricted to the numerical scheme. \demo
\end{remark}
\begin{remark}
Observe that our method gives us a constructive way to derive symplectic integrators for Hamiltonian systems. It will be interesting to compare our methods with other previous approaches (\cite{LeZa}), specially when the configuration space is a Lie group and we can use well-known retraction maps such as exponential maps and other approximations. See \cite{IsMuNoZa,BouMa,CMO,BogMa}. \demo
\end{remark}
\subsection{Geometric integrators in Lagrangian framework}\label{Sec:GeomIntLag}

Let us consider a regular Lagrangian function $L: TQ\rightarrow {\mathbb R}$ so that there exists a second-order vector field $\Gamma_L$ on $TQ$ and Euler-Lagrange equations are given by 
$${\rm i }_{\Gamma_L}\Omega_L={\rm d} E_L,$$
where $E_L$ is the energy function and $\Omega_L$ is the symplectic Lagrange 2-form obtained by the pull-back of the Legendre map ${\mathcal F}L\colon TQ \rightarrow T^*Q$ of the natural symplectic form on $T^*Q$, that is, $\Omega_L=({\mathcal F}L)^*\omega_Q$ (see \cite{AbMa} for more details).

As in Equation~\eqref{eq:2ndmethod-SODE}, a discretization map on $TQ$, that is, $R^{TTQ}_d\colon TTQ \rightarrow TQ\times TQ$, defines the following numerical integrator:
\begin{equation}\label{eq:LMethod}
 R^{TTQ}_d\left(h\, \Gamma_L\left(\left(\tau_{TQ}\circ \left(R_d^{TTQ}\right)^{-1}\right)(q_0,v_0;q_1,v_1)\right)\right)=(q_0,v_0;q_1,v_1).
\end{equation}
Equivalently, \begin{equation*}
h\, \Gamma_L\left(\left(\tau_{TQ}\circ \left(R_d^{TTQ}\right)^{-1}\right)(q_0,v_0;q_1,v_1)\right)=\left(R_d^{TTQ}\right)^{-1}(q_0,v_0;q_1,v_1).
\end{equation*}

As the Lagrangian function is regular, we could move to the Hamiltonian framework and construct a symplectic numerical integrator using the cotangent lift of a discretization map on $Q$ as in Proposition~\ref{Prop:HsymplecticMethod}. It remains to prove if the obtained numerical integrator is a discretization map on $TQ$ as described in Definition~\ref{def:DiscreteMap2}.

Remember that the manifold $(TQ\times TQ, \Omega_L^1-\Omega_L^0)$ is symplectic. Locally, the symplectic 2-form is given by $\Omega_L^1-\Omega_L^0=({\mathcal F}L,{\mathcal F}L)^*( {\rm d}q_i^1\wedge {\rm d} p_i^1\,- {\rm d}q_i^0\wedge {\rm d} p_i^0)$.

\begin{proposition} \label{Prop:LsymplecticMethod} Let $R_d\colon TQ\rightarrow Q\times Q$ be a discretization map on $Q$ and $L: TQ\rightarrow {\mathbb R}$ be a regular Lagrangian function. The two following facts are satisfied:
	\begin{enumerate}
		\item[(a)] the map $R_d^{L}=({\mathcal F}L^{-1},{\mathcal F}L^{-1})\circ R_d^{T^*}\circ T{\mathcal F}L\colon TTQ \rightarrow TQ\times TQ$ defines a symplectic numerical integrator of the Euler-Lagrange equations for $L$;
		\item[(b)] the above-mentioned map $R_d^{L}$ is a discretization map on $TQ$. 
	\end{enumerate}
	\end{proposition}
\begin{proof}
First, we prove property $(a)$. As the Lagrangian function is regular, the Legendre map is a local diffeomorphism. Propositions~\ref{Prop_CotLift} and~\ref{Prop:HsymplecticMethod} guarantee that $R_d^L$ is a symplectomorphism  because it is a composition of symplectomorphisms. Hence, $R_d^{L}$ defines a symplectic numerical integrator in Equation~\eqref{eq:LMethod}.

The diagram below shows the constructive process for $R_d^{L}$:

\begin{equation*}
\xymatrix{  T^*Q \times T^*Q \ar[rrr]^{({\mathcal F}L^{-1},{\mathcal F}L^{-1})} &&& TQ\times TQ
	 \\ &  TT^*Q \ar[lu]_{R_d^{T^*}} & TTQ \ar[l]_{T{\mathcal F}L}  \ar[ru]^{R_d^{L}} & Q\times Q \\ &  T^*Q \ar[u]^{X_H} & TQ \ar[ru]^{R_d} \ar[l]_{{\mathcal F}L}  \ar[u]_{\Gamma_L} & }
\end{equation*}

In other words, the Lagrangian submanifold ${\rm Im}\, \Gamma_L$ of $\left(TTQ,\left(T{\mathcal F}L\right)^*({\rm d}_T \omega_Q)\right)$ is preserved by $R_d^{L}$ and the numerical method in Equation~\eqref{eq:LMethod} is symplectic. 

For $(b)$, we must prove the properties in Definition~\ref{def:DiscreteMap2} for the map $R^L_d \colon TTQ \rightarrow TQ\times TQ$.

\begin{enumerate}
	    \item Note that $R^L_d(v_q,0_{v_q})=(v_q,v_q)$ because
	    \begin{align*}
	        R^L_d(v_q,0_{v_q})&=\left(({\mathcal F}L^{-1},{\mathcal F}L^{-1})\circ R_d^{T^*}\right)\left({\mathcal F}L(v_q); 0_{{\mathcal F}L(v_q)}\right)\\&=({\mathcal F}L^{-1},{\mathcal F}L^{-1})\left({\mathcal F}L(v_q); {\mathcal F}L(v_q)\right)\\&=(v_q,v_q)\, .
	    \end{align*}
	    The second equality is true because $R_d^{T^*}$ is a discretization map on $T^*Q$ as shown in Proposition~\ref{prop_liftT^*Rd}.
	    \item We must prove that $T_{(q,v,0,0)}\left(R^L_d\right)_{(q,v)}^2-T_{(q,v,0,0)}\left(R^L_d\right)_{(q,v)}^1$ is the identity map from $T_{(q,v,0,0)}TTQ\simeq T_{(q,v)}TQ$ to itself.
	    
	    Let us first compute it for $i=1,2$:
	    \begin{align*}
	        \left.\dfrac{\rm d}{{\rm d}\, t}\right\vert_{t=0}\,& {\mathcal F}L^{-1}\left( \left( R_d^{T^*} \right)^i \left(T {\mathcal F} L\right) (q,v,t\, \dot{q},t\, \dot{v})\right) \\=& D {\mathcal F}L^{-1}_{(q,\frac{\partial L}{\partial v})} T_{(q,\frac{\partial L}{\partial v},0,0)} \left( R_d^{T^*}\right)^i D_{3,4}\left(T {\mathcal F} L\right)_{(q,v,0,0)}
	    \end{align*}
	    
	    Note that  $D_{3,4}\left(T {\mathcal F} L\right)_{(q,v,0,0)}$ is the fiber derivative of the tangent map $T{\mathcal F} L$. Knowing that the tangent map is linear on the fiber, 
	    together with the fact that $R_d^{T^*}$ is a discretization map on $T^*Q$ and it satisfies the second property in Definition~\ref{def:DiscreteMap2}, we can conclude that the map  $R^L_d$ satisfies the second property for being a discretization map on $TQ$.
	\end{enumerate}
\end{proof}

It can be proved that only for a very specific discretization map $R_d$ on $Q$ and Lagrangian function, the discretization map $R^L_d$ of $TTQ$ is the tangent lift of $R_d$.

\begin{corollary}\label{Corol:Lkinetic} Let  $L(q,v)=\dfrac{1}{2}\, v^T \, M v-V(q)$ be the Lagrangian function, being $M$ a positive-definite symmetric constant mass matrix and $V$ the potential function. If $R_d\colon TQ\rightarrow Q \times Q$ is the discretization map on $Q$ given by the mid-point rule in Example~\ref{Ex:basic}, then $R^L_d$ is the tangent lift of $R_d$.
	\end{corollary}

\begin{proof}
The Legendre transformation for $L$ in the corollary is ${\mathcal F}L(q,v)=(q,Mv)$ and the inverse map is ${\mathcal F}L^{-1}(q,p)=(q,M^{-1}p)$. Thus,
	\begin{equation*}
	    R^L_d(q,v,\dot{q},\dot{v})=\left(R^1_d(q,\dot{q}), -(\dot{v},v){\rm D}R_d^{-1}(q,\dot{q})_{\ast,1};R^2_d(q,\dot{q}), (\dot{v},v){\rm D}R_d^{-1}(q,\dot{q})_{\ast,2} \right)
	\end{equation*}
	 where $A_{\ast,i}$ denotes the $i$th column of the matrix $A$. This expression is equal to the tangent lift $R_d^T$ if
	 \begin{align*}
	     -(\dot{v},v){\rm D}R_d^{-1}(q,\dot{q})_{\ast,1}&= D_{(q,v)}R_d^1(q,v) (\dot{q},\;\dot{v})^T\\
	     (\dot{v},v){\rm D}R_d^{-1}(q,\dot{q})_{\ast,2}&=D_{(q,v)}R_d^2(q,v) (\dot{q},\;\dot{v})^T\, .
	 \end{align*} Both equalities are satisfied if the discretization map $R_d$ is given by the mid-point rule, see Example~\ref{Ex:basic}.
\end{proof}

\begin{example}
	Let us consider a Lagrangian second-order vector field given by $\ddot{q}=-M^{-1} \, \nabla V(q)$. The numerical method in Equation~\eqref{eq:LMethod} for the mid-point rule described in Example~\ref{Ex:LiftBasic} becomes:

	\begin{align*}
	&R^{T}_d\left(h\, \Gamma_L\left( \dfrac{q_0+q_1}{2}, \dfrac{v_0+v_1}{2} \right)\right)=(q_0,v_0,q_1,v_1) \\ &
	R^{T}_d\left(\dfrac{q_0+q_1}{2}, \dfrac{v_0+v_1}{2}  , h(q_1-q_0), -h\, M^{-1}\, \nabla V\left(\dfrac{q_0+q_1}{2} \right) \right)=(q_0,v_0,q_1,v_1)\, .
		\end{align*}
		Given $(q_0,v_0)$, the numerical integrator is defined implictly by 
		\begin{eqnarray*}
			\dfrac{v_0+v_1}{2}&=&\dfrac{q_1-q_0}{h}\; ,\\
			\dfrac{v_1-v_0}{h}&=&-M^{-1}\, \nabla V\left(\dfrac{q_0+q_1}{2}\right) \, .
			\end{eqnarray*}
		After some straightforward computations, we obtain that this discrete method is rewritten as an implicit second order discrete equation given by: 
		\[
		\dfrac{q_2-2q_1+q_0}{h^2}=-\dfrac{1}{2}M^{-1}\left(\nabla V\left(\dfrac{q_0+q_1}{2}\right)+ V\left(\dfrac{q_1+q_2}{2}\right) \right)
		\]
		In the next subsection we will explore the relation of these  methods with discrete variational calculus.  \demo

\end{example}

 \subsection{Discrete variational calculus} \label{Sec:discreteVarCalc}

As mentioned in~\cite{MW_Acta}, a usual way to design symplectic integrators from a Lagrangian system consists of discretizing the variational principle using a  discrete Lagrangian map.  Many of these discrete maps are obtained from a continuous Lagrangian map and a discretization map on $Q$ by discretizing the continuous action as follows 
\[
{\mathcal S}(q_0, q_1)=\int_0^h L(q(t), \dot{q}(t))\; dt\approx h L\left(\frac{1}{h}R_d^{-1}(q_0, q_1)\right)=L^h_d(q_0,q_1)\, ,
\]
where $q(t)$ is the unique solution of the Euler-Lagrange equations such that $q(0)=q_0$ and $q(h)=q_1$ with $h$ enough small. Observe that if 
$R_d^{-1}(q_0, q_1)=v_q\in T_qQ$ then $\frac{1}{h}R_d^{-1}(q_0, q_1)=\frac{1}{h}v_q\in T_qQ$.
Therefore,  the discrete Lagrangian  $L^h_d: Q\times Q\rightarrow {\mathbb R}$ is defined by  $L^h_d=h\left(L\circ \frac{1}{h}R_d^{-1}\right)$.

If we  consider the Hamiltonian function $H(p,q)=\langle p, \dot{q} \rangle -L(q,\dot{q})$, then  we can simultaneously consider the discretization of both the Lagrangian and Hamiltonian framework. The following diagram is commutative by construction (see \cite{Tu} for the left-hand side of the diagram): 
\begin{equation*}
\xymatrix{  &&TT^*Q  &&& T^*Q\times T^*Q\ar@<-.5ex>[lll]_{\frac{1}{h}(R_d^{T^*})^{-1}} \\ 
&&T^*TQ\ar[u]^{\alpha_Q^{-1}}&&& T^*(Q\times Q)\ar[u]_{\Phi^{-1}}\ar@<-.5ex>[lll]_{\frac{1}{h}\widehat{R_d}^{-1}} \\
T^*Q\ar@<.5ex>[rruu]^{X_H}&&\ar@<-.5ex>[ll]_{{\mathcal F}L}TQ\ar[u]^{dL}&&&\ar@<-.5ex>[lll]_{\frac{1}{h}R^{-1}_d}Q\times Q
\ar[u]_{d L^h_d}
}
\end{equation*}
In this diagram we understand that the multiplication by $1/h$ in  $\frac{1}{h}R^{-1}_d$ is with respect to the vector bundle structure given by $\tau_Q: TQ\rightarrow Q$, $\frac{1}{h}\widehat{R_d}^{-1}$ with respect to the vector bundle structure given by $\pi_{TQ}: T^*TQ\rightarrow T^*Q$ and $\frac{1}{h}(R_d^{T^*})^{-1}$ with respect to $\tau_{T^*Q}: TT^*Q\rightarrow T^*Q$.

Therefore, we have that 
\begin{align}\label{eq:NumIntegratorL}
R_d^{T^*}\left( hX_H\left({\mathcal F} L \left(\frac{1}{h}R_d^{-1}(q_k, q_{k+1})\right)\right) \right)&=\Phi^{-1}\widehat{R_d}\left( h {\rm d} L\left(\frac{1}{h}( R_d^{-1})(q_k,q_{k+1})\right)\right)\nonumber\\
&=\Phi^{-1} \left( {\rm d}\, L^h_d (q_k, q_{k+1})\right)\, .
\end{align}

Using the previous equation and~\eqref{Eq:HMethod2} we obtain 
\begin{equation}
\label{Eq:HMethod2-3}
\Phi^{-1} \left( {\rm d}\, L^h_d (q_k, q_{k+1})\right)  = (q_k, p_{k}; q_{k+1}, p_{k+1})
\end{equation}
and the discrete variational  equations in~\cite{MW_Acta} are recovered:  
\begin{eqnarray*}
p_k&=&-D_1L^h_d(q_k, q_{k+1})\, ,\\
p_{k+1}&=&D_2L^h_d(q_k, q_{k+1})\, . 
\end{eqnarray*}
These equations lead to the well-known discrete Euler-Lagrange equations:  
\[
D_1L^h_d(q_k,q_{k+1})+D_2L^h_d(q_{k-1},q_k)=0\, .
\]

\begin{example} 
Given a regular Lagrangian $L: TQ\rightarrow {\mathbb R}$, where $Q$ is a vector space, 
consider the discrete Lagrangian $L^h_d(q_0, q_1)=h L\left( \frac{q_0+q_1}{2}, \frac{q_1-q_0}{h}\right)$.
Using  Example~\ref{Ex:LiftBasic}, the left-hand side of Equation~\eqref{eq:NumIntegratorL}  and the right-hand side of~\eqref{Eq:HMethod2-3}, we obtain the following integrator: 
\begin{eqnarray*}
\frac{p_{k+1}-p_k}{h}&=&\frac{\partial L}{\partial q}\left(\frac{q_k+q_{k+1}}{2}, \frac{q_{k+1}-q_k}{h}\right)\, ,\\
\frac{p_k+p_{k+1}}{2}&=&\frac{\partial L}{\partial \dot{q}}\left(\frac{q_k+q_{k+1}}{2}, \frac{q_{k+1}-q_k}{h}\right)\, .  \demo
\end{eqnarray*}

\end{example}

\section{Composition of geometric integrators}\label{Sec:Compose}

The construction of symplectic integrators based on discretization maps is closely related to the notion of Lagrangian submanifolds, as already appears in Section~\ref{Sec:GeomInt}. For instance,  Equation   (\ref{Eq:HMethod2}) defines the following Lagrangian submanifold of the symplectic manifold $(T^*Q\times T^*Q, \Omega_{12})$
\begin{equation}\label{lag-expression}
{\mathcal L}^h=\left\{
 (\alpha_q, \beta_{q'})\in T^*Q\times T^*Q\; \mid\; 
  (\alpha_q, \beta_{q'})=R_d^{T^*}( h\, X_H \left(\gamma_{q''}\right)) \right\}
\end{equation}
where $\gamma_{q''}=(\tau_{T^*Q}\circ \left(R^{T^*}_d\right)^{-1})(\alpha_q, \beta_{q'})\in T^*Q$.

Now, we will use some well-known properties of Lagrangian submanifolds as the composition of Lagrangian submanifolds (see~\cite{90GuiStern} for more details) to describe a particularly elegant method to construct high-order methods from a given low-order integrator (see \cite{hairer}). 
To be more precise, we are going to geometrically describe  the composition   of two (or more)  geometric integrators defined by different discretization maps. As a particular  example, we will recover the well-known St\"ormer-Verlet method, a second-order symplectic  method. 

Let $R_{d,1}$ and $R_{d,2}\colon TQ \rightarrow Q\times Q$ be two discretization maps on $Q$ and $H: T^*Q\rightarrow {\mathbb R}$ be a Hamiltonian function, using Equation (\ref{lag-expression}) we define two Lagrangian submanifolds of $(T^*Q\times T^*Q, \Omega_{12})$ as follows: 
\begin{align*}
{\mathcal L}^{h/2}_1=&\bigg\{
 (q_k, p_k;q_{k+1/}, p_{k+1/2})\in T^*Q\times T^*Q\; \mid\; \exists \; \gamma_{k, k+{1/2}}\in T^*Q \mbox{ s. t. }  \\ & 
  (q_k, p_k ;  q_{k+1/2}, p_{k+1/2})=R_{d,1}^{T^*}\left( \frac{h}{2}\, X_H \left(\gamma_{k, k+{1/2}}\right)\right) \bigg\}\, ,\\
 {\mathcal L}^{h/2}_2= & \bigg\{
 (q_{k+1/2}, p_{k+1/2}; q_{k+1}, p_{k+1})\in T^*Q\times T^*Q\; \mid\;  \exists \; \gamma_{k+{1/2}, k+1}\in T^*Q \mbox{ s. t.}  \\
  & (q_{k+1/2}, p_{k+1/2}; q_{k+1}, p_{k+1})=R_{d,2}^{T^*}\left( \frac{h}{2}\, X_H \left(\gamma_{ k+{1/2}, k+1}\right)\right) \bigg\} \, ,
\end{align*}
where 
\begin{align*}\gamma_{k, k+{1/2}}&=\left(\tau_{T^*Q}\circ \left(R^{T^*}_{d,1}\right)^{-1}\right)(q_k, p_k; q_{k+1/2}, p_{k+1/2})\in T^*Q\, , \\
\gamma_{ k+{1/2}, k+1}&=\left(\tau_{T^*Q}\circ \left(R^{T^*}_{d,2}\right)^{-1}\right)(q_k, p_k; q_{k+1/2}, p_{k+1/2})\in T^*Q\, .
\end{align*}

Under the assumption of clean intersection (see \cite{13GuiStern}), we compose the above two Lagrangian submanifolds as follows
\begin{align*}
{\mathcal L}^{h/2}_2\circ {\mathcal L}^{h/2}_1=& \left\{(\alpha_q,\beta_{q''})\in T^*Q\times T^*Q\; \mid\, \exists\;  \gamma_{q'}\in T^*Q \hbox{ with }
(\alpha_q,\gamma_{q'})\in {\mathcal L}_1^{h/2}, \right. \\ & \left. (\gamma_{q'},\beta_{q''})\in {\mathcal L}_2^{h/2}\right\}\, ,
\end{align*}
obtaining an immersed Lagrangian submanifold. Thus, it generates a new symplectic integrator. Moreover, it is possible to compose more than two Lagrangian submanifolds to generate more involved methods where the intermediate points $\gamma_{q'}$ play the role of micro-nodes (see \cite{MW_Acta,Leok-Shingel, cedric}). 
\begin{example}
Let $Q={\mathbb R}^n$, we consider the two discretization maps $R_{d,1}(q, v)=(q, q+v)$ and $R_{d,2}(q, v)=(q-v, q)$. Then, we compute their corresponding cotangent lifts, $R^{T^*}_{d,1}$ and $R^{T^*}_{d,2}$ as described in Section~\ref{Sec:coTliftRd}, and obtain
\begin{eqnarray*}
{\mathcal L}^{h/2}_1&=&\left\{
 (q_k, p_k; q_{k+1/2}, p_{k+1/2})\;\left\vert \; 
  \begin{array}{c}
 p_{k+1/2}=p_k-\frac{h}{2} \nabla_q H(q_k, p_{k+1/2})\\ q_{k+1/2}=q_k+\frac{h}{2}\nabla_p H(q_k, p_{k+1/2})      
  \end{array} \right.\right\}\\
 {\mathcal L}^{h/2}_2&=&\left\{
 (q_{k+1/2}, p_{k+1/2}; q_{k+1}, p_{k+1})\; \left\vert\; 
 \begin{array}{c} p_{k+1}=p_{k+1/2}-\frac{h}{2} \nabla_q H(q_{k+1}, p_{k+1/2})\\ q_{k+1}=q_{k+1/2}+\frac{h}{2} \nabla_p H(q_{k+1}, p_{k+1/2})
 \end{array} \right.
 \right\}
\end{eqnarray*}
The composition ${\mathcal L}^{h/2}_2\circ {\mathcal L}^{h/2}_1$ gives a new symplectic integrator that corresponds with the St\"ormer-Verlet method \cite{hairer}: 
\begin{eqnarray*}
p_{k+1/2}&=&p_k-\frac{h}{2} \nabla_q H(q_k, p_{k+1/2})\, ,\\
q_{k+1}-\frac{h}{2} \nabla_p H(q_{k+1}, p_{k+1/2})&=&q_k+\frac{h}{2}\nabla_p H(q_k, p_{k+1/2})\, ,\\
p_{k+1}&=&p_{k+1/2}-\frac{h}{2} \nabla_q H(q_{k+1}, p_{k+1/2})\, . \demo
\end{eqnarray*}

\end{example}

When discrete Lagrangian functions are given as in Section~\ref{Sec:discreteVarCalc}, Equations~\eqref{Eq:HMethod2} and~\eqref{Eq:HMethod2-3} can be expressed as Lagrangian submanifolds of $(T^*Q\times T^*Q, \Omega_{12})$ and many of the methods described in \cite{MW_Acta,Leok-Shingel} are recovered.

For instance, for a small positive step size $h$, we consider the following two discretization maps on $Q$, $R_{d,i}\colon TQ \rightarrow Q\times Q$:
\begin{eqnarray*}
R_{d,1}\left(q, \frac{h}{2}v\right)=\left(q, q+\frac{h}{2}v\right), & \mbox{ with inverse } &  R_{d,1}^{-1}(q_0,q_1)=\left(q_0, \dfrac{q_1-q_0}{h/2}\right)\, , \\
R_{d,2}\left(q, \frac{h}{2}v\right)=\left(q-\frac{h}{2}v, q\right), & \mbox{ with inverse } & R_{d,2}^{-1}(q_0,q_1)=\left(q_1, \dfrac{q_1-q_0}{h/2}\right)\, . \end{eqnarray*} We define the discrete Lagrangian functions $L_{d,i}=\left(h\, L\circ \left(R_{d,i}\right)^{-1}\right)\colon Q\times Q \rightarrow \mathbb{R}$ such that the image of $(\Phi^{-1}\circ \, {\rm d} L_{d,i})$ define the Lagrangian submanifolds ${\mathcal L}_i$ of the symplectic manifold $(T^*Q \times T^*Q, \Omega_{12})$. The submanifolds ${\mathcal L}_i$ define a discrete dynamical system whose equations are locally described by
$${\mathcal L}_i=\{(q_0,p_0;q_1,p_1) \in T^*Q \times T^*Q \; \mid \; p_0=-{\rm D}_1 {\rm L}^i_d(q_0,q_1), \quad   p_1={\rm D}_2 {\rm L}^i_d(q_0,q_1)\}\, .$$
The  composition 
\begin{equation*}
{\mathcal L}_{2}\circ {\mathcal L}_1=\{(\alpha_1,\alpha_2)\; \mid \; \exists \; \alpha_{1/2}\in T^*Q \mbox{ s. t. } (\alpha_1,\alpha_{1/2})\in {\mathcal L}_1 \, , \,  (\alpha_{1/2},\alpha_2)\in {\mathcal L}_2\}\, .
\end{equation*}
has associated the dynamics given by the discrete Lagrangian ${\rm L}^3_d(q_0,q_2)={\rm L}^1_d(q_0,q_1)+{\rm L}^2_d(q_1,q_2)$, that plays the role of generating function (see also \cite{Fernando}). The discrete equations are
\begin{eqnarray*}
	p_0&=&-{\rm D}_1 {\rm L}_d^1(q_0,q_1)\, ,\\
	0&=&{\rm D}_2 {\rm L}^1_d(q_0,q_1)+{\rm D}_1 {\rm L}_d^2(q_1,q_2)\, ,\\
	p_2&=&{\rm D}_2 {\rm L}^2_d(q_1,q_2)\, . 
\end{eqnarray*}

\subsection{Symplectic symmetric methods}\label{section-ssm}

If we have a Lagrangian submanifold ${\mathcal L}$ of $(T^*Q\times T^*Q, \Omega_{12})$, then the transpose ${\mathcal L}^{\dagger}$ defined by
\[
{\mathcal L}^{\dagger}=\{(\alpha_q, \beta_{q'})\in T^*Q\times T^*Q\; \mid\; (\beta_{q'}, \alpha_q)\in {\mathcal L}\}
\]
is also a Lagrangian submanifold of $(T^*Q\times T^*Q, \Omega_{12})$. 

For a Hamiltonian function and a discretization map on $T^*Q$, we consider the following Lagrangian submanifold  used in the previous section:
\[
{\mathcal L}^h=\left\{
 (\alpha_q, \beta_{q'})\in T^*Q\times T^*Q\; \mid\; \exists \; \gamma_{q''}\in T^*Q \mbox{ s. t. }
  (\alpha_q; \beta_{q'})=R_d^{T^*}( h\, X_H \left(\gamma_{q''}\right)) \right\}\, .
\]
As described in~\cite{hairer, MW_Acta}, the composition of symplectic methods (seen here as  Lagrangian submanifolds) gives rise to new symplectic methods. For instance, the Lagrangian submanifold \[
\left({\mathcal L}^{h/2}\circ {\mathcal L}^{h/2}\right)^{\dagger}\, 
\]
is another way to interpret the St\"ormer-Verlet method considered in the previous section.

\begin{definition}
A symplectic method  defined by ${\mathcal L}^h$ is {\bf symmetric}  if \[
\left({\mathcal L}^h\right)^{\dagger}={\mathcal L}^{-h}\, .
\]
\end{definition}

\begin{proposition} Let $\iota: Q\times Q\rightarrow Q\times Q$ be the inversion map defined by $\iota (q, q')=(q', q)$.  
If $R_d(v_q)=\iota(R_d(-v_q))$ for all $v_q\in T_qQ$, then ${\mathcal L}^h$ is symmetric.
\end{proposition}
\begin{proof}
Observe that 
\[
({\mathcal L}^h)^{\dagger}=\left\{
  (\alpha_q,\beta_{q'})\in T^*Q\times T^*Q\; \mid\; \exists \; \gamma_{q''}\in T^*Q \mbox{ s. t. }
  (\beta_{q'},\alpha_q)=R_d^{T^*}( h\, X_H \left(\gamma_{q''}\right))\right\}\, .
\]
By Proposition~\ref{prop_cotangent_sym}, $R_d^{T^*}$ is symmetric and if we apply the inversion ${\iota}_{T^*Q}: T^*Q\times T^*Q\rightarrow T^*Q\times T^*Q$ on $T^*Q$ on both sides of the equation defining $({\mathcal L}^h)^{\dagger}$ we get:
$$\iota_{T*Q}(\beta_{q'},\alpha_q)=\iota_{T^*Q}\left(R_d^{T^*}( h\, X_H \left(\gamma_{q''}\right))\right)=R_d^{T^*}( -h\, X_H \left(\gamma_{q''}\right)).$$
Thus, we immediately deduce that $({\mathcal L}^h)^{\dagger}={\mathcal L}^{-h}$. 
\end{proof}

It is well-known that the order of a symmetric method is necessarily even, then using symmetric discretization maps as in Definition \ref{symmetric-discretization} we always obtain a second-order method.

\subsection{Construction of higher-order symplectic methods} \label{Sec_higherorder}

In the previous sections we have introduced  first and second-order symplectic methods starting with different discretization maps.  Now, we will show that the composition of Lagrangian submanifolds is a geometric tool to produce higher-order symplectic methods  equivalent to the composition of numerical methods (see \cite{hairer,LeRe,Yoshida,blanes}) 

From an initial  discretization map $R_d: TQ\rightarrow Q\times Q$ and a Hamiltonian system $X_H$ we construct the Lagrangian submanifold ${\mathcal L}^h$ as in Equation~\eqref{lag-expression}. For real numbers $\gamma_1, \ldots, \gamma_s$, we define the Lagrangian submanifold
\begin{equation}\label{comp}
{\mathcal L}^{\gamma_sh}\circ \ldots \circ {\mathcal L}^{\gamma_1h}
\end{equation}
that generates a symplectic composition method. 

If ${\mathcal L}^h$ generates a method of order two  and the coefficients $\gamma_1, \ldots, \gamma_s$ verify 
\begin{align*}
\gamma_1+ \ldots + \gamma_s = 1\, ,\\
\gamma^3_1+ \ldots + \gamma^3_s = 0,
\end{align*}
then the  symplectic composition method~\eqref{comp}  is at least of order 3.

As described in~\cite{hairer} for $s=3$, if we start with a a method of order two
\[
\gamma_1 =\gamma_3= \frac{1}{2-2^{1/3}}, \quad  
\gamma_2= -\frac{2^{1/3}}{2-2^{1/3}}\,  ,
\]
we obtain a method of order 4 due to the symmetry of the coefficients. By repeating this procedure we obtain methods of order 6, 8, etc. Of course, other choices of the parameters produce different higher-order numerical methods (see for instance \cite{Suzuki,Mclachan}). 
Additionally, other techniques like splitting methods perfectly fit in our framework as composition of Lagrangian submanifolds.

Another interesting family of methods are the symplectic Runge-Kutta methods. For instance, the diagonally implicit Runge-Kutta methods (that is, the coefficients verify $a_{ij}=0$ if $i<j$) and  $b_i\not=0$ are derived from the symmetric discretization map
$R_d(q, v)=(q-v/2, q+v/2)$.
For a Hamiltonian function $H: T^*Q \rightarrow {\mathbb R}$ the map $R_d$ produces the Lagrangian submanifold of $T^*Q\times T^*Q$:
\[
{\mathcal L}^h=\left\{
 (q_k, p_k, q_{k+1}, p_{k+1})\in T^*Q \times T^*Q \; \left\vert\; \
\begin{array}{r}
\frac{q_1-q_0}{h}=\frac{\partial H}{\partial p}\left(\frac{q_0+q_1}{2}, \frac{p_0+p_1}{2}\right)\, \\ \\
\frac{p_1-p_0}{h}=
-\frac{\partial H}{\partial q}\left(\frac{q_0+q_1}{2}, \frac{p_0+p_1}{2}\right)\, 
\end{array}\right.
 \right\}\, .
\]
This Lagrangian submanifold corresponds with the implicit midpoint rule. It is well-known that the symplectic diagonally implicit Runge Kutta methods
are equivalent to the following composition  \cite{hairer}
\[
{\mathcal L}^{b_1h}\circ {\mathcal L}^{b_2h}\circ \ldots {\mathcal L}^{b_sh}\, .
\]
Similar argument also holds for partitioned  Runge–Kutta method based on two diagonally implicit methods

To sum up, our constructions allow to reinterpret other well-known techniques for designing higher-order methods. Specifically,  from second-order methods obtained by discretization maps in Section~\ref{section-ssm} we can derive higher-order methods using composition of Lagrangian submanifolds. These techniques are not limited to standard Hamiltonian systems but can also be used for Hamiltonian systems defined on cotangent bundles of manifolds (as Lie groups \cite{BogMa} for instance) or more general situations as we will describe in Section~\ref{Sec:future}.

\section{Conclusions and future work}\label{Sec:future}
In this paper we have introduced the lift of a discretization map to tangent and cotangent bundles which are the  phase spaces of mechanical systems. 
These lifts allow us to derive geometric integrators for systems defined by a Lagrangian or Hamiltonian function. Standard constructions in symplectic geometry, as well as properties of Lagrangian submanifolds, create a geometric framework to obtain several well-known symplectic integrators. Our geometric point of view opens the door for new types  of applications of discretization maps, as well as, for the construction of geometric (symplectic) integrators following simple rules  (lifting of retractions, composition and generating functions for Lagrangian submanifolds, etc). Now, we will mention some promising future research lines.    

\subsection{Reduced systems and systems with holonomic constraints }
The notion of  lift of  discretization maps can be easily extended to the Lie algebroid setting using the Lie algebroid prolongation \cite{LeMaMa}. This theory covers all the examples of reduced systems  by symmetry groups. Therefore, combining both constructions we can  directly apply our results to the construction of geometric integrators  for Lagrangian or Hamiltonian functions invariant under the action of a symmetry Lie group \cite{MR1365779,MR2230001}. It is also well-known how to produce geometric integrators for systems subject to holonomic constraints. Thus, it will be interesting to produce  constrained geometric integrators using our approach and compare them with \cite{LeRe,McMoVe}.

\subsection{Discrete gradient methods}
In general for ordinary differential equations in ${\mathbb R}^n$ in skew-gradient form, i.e. $\dot{x}=\Pi(x)\nabla H(x)$ where $x\in \mathbb{R}^n$ and $\Pi(x)$ is a skew-symmetric matrix, it is clear that  $H$ is a first integral.  Using discretizations of the gradient $\nabla H(x)$ it is possible to define a class of integrators that preserve exactly the first integral $H$ (see \cite{Gonzalez,McLachlan}).
They are defined as follows:  
let $H:\mathbb{R}^n\rightarrow \mathbb{R}$ be a differentiable function, then $	\overline{\nabla}H:\mathbb{R}^{2N}\longrightarrow \mathbb{R}^N$ is a discrete gradient of $H$ if it is continuous and satisfies
		\begin{align}	
			\overline{\nabla}H(x,x')^T(x'-x)&=H(x')-H(x)\, , \quad \, \mbox{ for all } x,x' \in\mathbb{R}^n  \, , \nonumber \\
				\overline{\nabla}H(x,x)&=\nabla H(x)\, , \quad \quad \quad \quad \mbox{ for all } x \in\mathbb{R}^n  \, . \label{discGrad}
		\end{align}
For a Hamiltonian system $H: T^*Q\rightarrow {\mathbb R}$, we can generalize the previous construction by using a discretization map $R_d^{TT^*Q}: TT^*Q \rightarrow T^*Q\times T^*Q$ on a  general differentiable manifold $T^*Q$. We define a discrete gradient as a map
$\overline{{d}{H}}: T^*Q\times T^*Q \longrightarrow T^{*}T^*Q$ that makes the following diagram commutative
$$
\xymatrix{
T^*Q\times T^*Q \ar[rr]^{\overline{{d}{H}}} \ar[d]_{\left(R_d^{TT^*Q}\right)^{-1}} && T^{*}T^*Q \ar[d]^{\pi_{T^*Q}}\\
	TT^*Q \ar[rr]^{\tau_{T^*Q}} && T^*Q
}
$$

Similar to~\eqref{discGrad},  $\overline{d}{H}$ must verify the following properties: 
		\begin{align*}
			\langle \overline{{d}{H}}(x,x'), \left(R_d^{TT^*Q}\right)^{-1}(x, x')\rangle &=H(x')-H(x)\, , \quad \, \mbox{ for all } x,x' \in T^*Q  \, , \\
			\overline{{d}}H(x,x)&=d H(x)\, , \quad \quad \quad \quad \mbox{ for all } x \in T^*Q  \, . 
		\end{align*}
In this case, an energy preserving integrator would be 
\[
\left(R_d^{TT^*Q}\right)^{-1}(x, x')=\omega^{\sharp}(x'')(\overline{{d}{H}}(x, x'))
\]
where $x''=\tau_{T^*Q}^*\left(\left(R_d^{TT^*Q}\right)^{-1}(x, x')\right)$. We will explore this possibility in a future paper (see also \cite{CMO,Celledoni},  for the case of extension to manifolds, in special, Riemannian manifolds).

\subsection{Higher-order lagrangian systems} \label{Future_higherorder}

Another topic of interest is related with the discretization of higher-order Lagrangian systems $L: T^{(k)}Q\rightarrow {\mathbb R}$ using appropriate higher-order lifts of discretization maps. These constructions will be useful for interpolation problems on manifolds and for optimal control problems (see \cite{Crouch,Invariant1,CFM}). For instance some interesting optimal control problems for mechanical systems are describe using a   second-order Lagrangian system $L: T^{(2)}Q\rightarrow {\mathbb R}$ (position, velocities and accelerations).  The corresponding hamiltonian formulation is given in the  cotangent bundle $T^*T Q$. To obtain a symplectic integrator we could use first the  tangent lift of a discretization map $R_d: TQ\rightarrow Q\times Q$ and finally its cotangent lift.

\subsection{Higher-order geometric methods}

Section~\ref{Sec_higherorder} is the starting point of a research line that seeks to obtain higher-order geometric methods by using discretization maps. 
In this paper, we have only briefly discussed composition and splitting methods, but in the future, it would be interesting to explore other possibilities, as for instance, the well-known family of implicit symplectic Runge-Kutta methods and symplectic partitioned Runge–Kutta method. Moreover, the use of discretization maps on the cotangent bundle that are not lifted from the base manifold may be useful for obtaining higher-order methods.

\subsection{Geometric integration of Dirac systems}
Dirac structures  were introduced in~\cite{CouWei,Courant} as a way to unify presymplectic and Poisson geometries giving a way to collect in the same geometric framework many situations of interest in mechanics and mathematical physics. As an example we can think of the Dirac structure $D\subset TT^*Q\oplus T^*T^*Q$ induced by the canonical symplectic structure $\omega_Q$,
that is 
\[
D=\{(X_{\alpha_q}, \lambda_{\alpha_q})\in TT^*Q\oplus T^*T^*Q\; \mid\; i_{X_{\alpha_q}}\omega_Q=\lambda_{\alpha_q}\}\, .
\]
For a Hamiltonian system  $H\colon T^*Q\rightarrow {\mathbb R}$, we can write Hamilton's equations as
\begin{equation}\label{eq:DiracHamEq}
  X_{\alpha_q}\oplus {\rm d} H(\alpha_q)\in D_{\alpha_q}\, ,
\end{equation}
with $\alpha_q\in T_q^*Q$.
As studied in the literature, Dirac structures are more general than the above example. They can be given by a Poisson tensor, for instance, or could not satisfy the integrability condition admitting new generalizations as in the case of nonholonomic constraints  (see \cite{barbero}). Moreover, it is also interesting to study the case of Dirac systems where the dynamics is not induced by a function on the cotangent bundle (as in the case of standard Hamiltonian dynamics), but for a general Lagrangian submanifold ${\mathcal S}$ of $(T^*T^*Q, \omega_{T^*Q})$ (as in the case of singular Lagrangians, optimal control theory, etc). Now, Equation~\eqref{eq:DiracHamEq} must be replaced by 
\begin{equation*}
  X_{\alpha_q}\oplus S_{\alpha_q}\in D_{\alpha_q}\, .
\end{equation*}
discretization maps could also be used here to deduce geometric integrators. Let $R_d: TQ\rightarrow Q\times Q$ be a discretization map on $Q$, the cotangent lift of $R_d$ defines the following geometric integrator 
\[
 \left(R_d^{T^*}\right)^{-1}(q_k, p_k; q_{k+1}, p_{k+1})  \oplus S_{\gamma_{k, k+1}}\in D_{\gamma_{k, k+1}}
\]
where $\tau_{T^*Q}\left( (R_d^{T^*})^{-1}(q_k, p_k; q_{k+1}, p_{k+1})\right)=\gamma_{k, k+1}$. We will study in a forthcoming paper the design of Dirac integrators using the lift of the discretization map, their geometrical properties (preservation of the associated presymplectic foliation, etc.) and compare them with other approaches on this topic (\cite{LeOh, Leok-Shingel}).

\subsection{Morse families for Lagrangian submanifolds and symplectic integration}
In this paper it is clear the close relationship between the design of different symplectic methods and the  construction of Lagrangian submanifolds. 
The notion of a Morse family or
phase function was introduced in~\cite{Hormander} (see also \cite{Weinstein-Alan}) and it is possible to prove that locally any Lagrangian submanifold is the image of a Lagrangian immersion generated by a Morse family. In a recent paper \cite{barbero}, we have combined Dirac structures and Morse families to obtain a geometric formalism that unifies most of the scenarios in mechanics (constrained calculus, nonholonomic systems, optimal control theory, higher-order mechanics, etc.), as the examples in the paper show. Employing the techniques introduced here we aim to study the construction of geometric integrators for all the above-mentioned cases, as well as, the notion of Morse family to construct new geometric integrators.

\section*{Acknowledgments}
The authors acknowledge financial support from the Spanish Ministry of Science and Innovation, under grants PID2019-106715GB-C21,   the Spanish National Research Council, through the ``Ayuda extraordinaria a Centros de Excelencia Severo Ochoa'' R\&D  (CEX2019-000904-S) and I-Link Project (Ref: linkA20079) from CSIC (CEX2019-000904-S). MBL has been financially supported by ``Programa propio de I+D+I de la Universidad Polit\'ecnica de Madrid: Ayudas dirigidas a j\'ovenes investigadores doctores para fortalecer sus planes de investigaci\'on".

\bibliography{References}

\end{document}